\documentclass{article}

\usepackage{arxiv}

\usepackage[utf8]{inputenc} 
\usepackage[T1]{fontenc}    
\usepackage{hyperref}       
\usepackage{url}            
\usepackage{booktabs}       
\usepackage{amsfonts}       
\usepackage{nicefrac}       
\usepackage{microtype}      
\usepackage{lipsum}		
\usepackage{graphicx}

\usepackage{amsthm}
\usepackage{graphicx}
\usepackage{amssymb,amsmath}
\usepackage{subfigure}
\usepackage{multirow}
\usepackage{hyperref}
\usepackage{color,soul}
\usepackage{url}
\usepackage{algorithm}
\usepackage{rotating}
\usepackage{array}

\newtheorem{remark}{Remark}
\newtheorem{theorem}{Theorem}

\usepackage{algpseudocode}
\DeclareMathOperator\supp{supp}
\DeclareMathOperator\dom{dom}
\DeclareMathOperator\argmin{arg\,min}
\DeclareMathOperator\minimize{minimize}

\newtheorem{assumption}{Assumption}

\usepackage{longtable}
\usepackage{pdflscape}
\usepackage{supertabular,booktabs}

\title{Online Stochastic DCA with applications to Principal Component Analysis \thanks{This work has been submitted to the IEEE for possible publication. Copyright may be transferred without notice, after which this version may no longer be accessible.}}

\author{ Hoai An Le Thi\\
	Universit\'e de Lorraine, LGIPM\\
	F-57000 Metz, France \\
	\texttt{hoai-an.le-thi@univ-lorraine.fr} \\
	\And
	Hoang Phuc Hau Luu \\
	Universit\'e de Lorraine, LGIPM\\
	F-57000 Metz, France \\
	\texttt{hoang-phuc-hau.luu@univ-lorraine.fr} \\
	\AND
	Tao Pham Dinh\\
	Laboratory of Mathematics, INSA-Rouen \\
	University of Normandie \\
	76801 Saint-\'Etienne-du-Rouvray Cedex, France\\
	\texttt{pham@insa-rouen.fr}
}

\date{}

\renewcommand{\headeright}{}
\renewcommand{\undertitle}{}
\renewcommand{\shorttitle}{Online Stochastic DCA with applications to Principal Component Analysis}

\begin{document}
\maketitle 

\begin{abstract}
Stochastic algorithms are well-known for their performance in the era of big data. In convex optimization, stochastic algorithms have been studied in depth and breadth. However, the current body of research on stochastic algorithms for nonsmooth, nonconvex optimization is relatively limited. In this paper, we propose new stochastic algorithms based on DC (Difference of Convex functions) programming and DCA (DC Algorithm) - the backbone of nonconvex, nonsmooth optimization. Since most real-world nonconvex programs fall into the framework of DC programming, our proposed methods can be employed in various situations, in which they confront stochastic nature and nonconvexity simultaneously. The convergence analysis of the proposed algorithms is studied intensively with the help of tools from modern convex analysis and martingale theory. Finally, we study several aspects of the proposed algorithms on an important problem in machine learning: the expected problem in Principal Component Analysis.
\end{abstract}


\section{Introduction}
\label{submission}
We consider the following optimization problem
\begin{equation}
\label{eq:ori}
\min_{w \in S}\{F(w) = \mathbb{E}(g(w,Z)) - \mathbb{E}(h(w,Z))\},
\end{equation}
where $S \subset \mathbb{R}^m$ is a nonempty, compact, and convex set, $Z$ is a random vector determined in some complete probability space $(\Omega, \mathcal{M}, \mathbb{P})$ such that $Z: \Omega \to \mathbb{R}^n$ and $g, h$ are functions satisfying some conditions described later. Broadly, $g$ and $h$ are those that make $G(w) = \mathbb{E}(g(w,Z))$ and $H(w) = \mathbb{E}(h(w,Z))$ convex, lower semi-continuous.

The framework of the problem (\ref{eq:ori}) is very general in two aspects. Firstly, the underlying distribution of $Z$ is arbitrary, which makes it able to treat any random variable involved. As a special case, when $Z$ is uniformly distributed over a finite set, we obtain a large-sum problem,
\begin{align}
\label{eq:empi}
\min_{w \in S} \left \{F(w) = \dfrac{1}{N}\sum_{i=1}^N{g(w,z_i)} - \dfrac{1}{N}\sum_{i=1}^N{h(w,z_i)} \right\}.
\end{align}
Secondly, in our setting, $g$ and $h$ are allowed to be nonsmooth, resulting in a very large class of stochastic nonsmooth, nonconvex DC programs which comprises most real-world problems \cite{lethi18}. Various learning problems possess DC structures, here we name a few: robust learning \cite{xu2020learning,collobert2006trading}, robust phase retrieval \cite{davis2019stochastic}, Positive Unlabeled (PU) learning with convex loss \cite{kiryo2017positive}, Difference of Log-sum-exp neural networks \cite{calafiore2020universal,bruggemann2020use}, principal component analysis \cite{montanari15}, etc.

Having said that, the main challenge of the problem (\ref{eq:ori}) also comes from the nonconvex structure of $F$ and the unknown underlying distribution of $Z$. So far, there is very few algorithms for stochastic nonconvex and nonsmooth problems of the general setting (\ref{eq:ori}).

In the literature, stochastic optimization has been investigated thoroughly for convex problems since the seminal work \cite{robbins51}. In this work, the authors introduced a novel idea of using stochastic approximations (SA) that results in Stochastic Gradient Descent (SGD). Thanks to its inexpensive computation cost, the SGD really opened a door in numerical optimization for large-scale problems \cite{bottou18,bottou10}. Hitherto, many variants of the SGD have been studied including stochastic subgradient descent \cite{ermoliev83,ruszczynski1986convergence}, incorporating Nesterov's acceleration technique \cite{ghadimi12}, using second-order information \cite{bottou04,bottou08,byrd16}. In nonconvex optimization, stochastic algorithms remain rare. Most of them require the objective to be smooth or partially smooth (some components of the objective are smooth). We list here some main approach to tackle nonconvex stochastic problems. Inspired by the aforementioned SGD, the first approach is stochastic (proximal) (sub)gradient-based methods which are mainly developed for smooth or weakly convex objective functions \cite{ghadimi2013stochastic,bertsekas2000gradient,davis2019stochastic}. In this approach, a gradient-like update is performed at each iteration where the proximal operator can be employed. The second is stochastic MM (Majorization-Minimization) for partially smooth objective \cite{mairal2013stochastic,razaviyayn2016stochastic}, in which the stochastic convex surrogate is constructed at each iteration and is minimized to obtain an updated optimization variable. The third is stochastic Successive Convex Approximation \cite{scutari2013decomposition,Yang16} (mainly for smooth objective functions) that is similar to stochastic MM where the sequence of approximation functions are convex but need not be the upper bound of sample objective functions. The fourth is stochastic DCA that aims to deal with stochastic DC programs - a substantially large class to cover almost all real-world nonconvex optimization problems \cite{lethi18}. Initial works in this approach include \cite{le2017stochastic,le2020stochastic,liu2020two,nitanda2017stochastic,xu19} that consider some special classes of DC problems such as large-sum and/or (partially) smooth, as well as \cite{an2020stochastic} working on a very general class of stochastic nonsmooth DC programs. To extend beyond the DC programming framework, \cite{Metel19} used Moreau envelope which is a DC function to approximate a nonsmooth, nonconvex regularizer, and then developed a stochastic DCA for solving the resulting problem. It is worth mentioning that in \cite{mairal2013stochastic,razaviyayn2016stochastic}, the authors also consider DC surrogates whose the second DC component is differentiable. It should be further noted that, as indicated in \cite{lethi18}, while the (stochastic) MM proposes a general idea to majorize the objective function, (stochastic) DCA gives the simplest and the most closed convex surrogate thanks to DC structures of the objective. Futhermore, usual choices of surrogates of MM result in DCA versions \cite{lethi18}.

In deterministic optimization context, DC  programming and DCA  
constitute a quite logical and natural extension of modern convex
analysis/programming to nonsmooth nonconvex analysis/programming,
sufficiently large to cover most real-world nonsmooth nonconvex programs,
but not too broad in order to explore/exploit the powerful arsenal of convex
analysis/programming. This theoretical and algorithmic philosophy was first
introduced in 1985 by Pham Dinh Tao, and widely developed by Le Thi Hoai An
and Pham Dinh Tao since 1993 to become now classic and increasingly popular
(see \cite{lethi05,phamdinh97,PLT98,PLT14}   and a comprehensible review on thirty years of developments
of DC  programming and DCA  in \cite{lethi18}).
It is widely recognized that DCA is one of rare algorithms to efficiently solve large-scale nonconvex and nonsmooth programs \cite{lethi18}. Thanks to the pervasiveness of DC programming and the flexible principle of DC reformulations, DCA   recover almost all standard methods in convex and nonconvex programming. Also, the flexibility and simplicity of DCA make the method a powerful tool to be employed in various applications in applied sciences including transport logistic, finance, computational biology, computational chemistry, robotics, data mining and machine learning, image processing and computer vision, cryptology, inverse problems and ill-posed problems, etc., see e.g., \cite{hoai2000efficient,Lethi:chap:01b,an2003large,lethi05,Lethi:adac:08,Lethi:jfcst:09,le2011solving,letmpj12,le2015dc,LeThi2016,pang2017computing,phamdinh97,lethi18} and the list of references in \cite{lethi18}.

To our knowledge, the paper \cite{an2020stochastic} is the first work dealing with the general setting (\ref{eq:ori}) where both DC components are allowed to be nonsmooth. In that article, the authors proposed several stochastic DCA schemes in the aggregated update style. That is, all past information (sample realizations) is used to construct subproblems. These algorithms therefore need to store all samples in the computer memory during the computational process. In this work, we investigate online stochastic DCA for the general problem (\ref{eq:ori}) to deal with fast  streaming data where we do not need to store samples all the time. Furthermore, thanks to the online mechanism, our proposed algorithms have the adaptive ability which is a great advantage over the SDCA schemes proposed in \cite{an2020stochastic}. Numerical experiments will justify this claim.

\emph{Paper's contribution.} 
We design three new online stochastic DCA schemes for solving the generic problem (\ref{eq:ori}) (which will be described in more details in section \ref{sec:osDCA}). The problem is very general in such a way that both DC components are nonsmooth. Besides, we will see that the assumptions used are mild that make the considered problem cover a very large class of real-world applications. Since the update steps of the proposed algorithms require \emph{new fresh samples} from the distribution of $Z$, we refer to our algorithms as \emph{online stochastic DCA} (osDCA in short). The first osDCA scheme constructs stochastic approximations (SA) for both values of $G$ and subgradients of $H$. The convergence analysis of the proposed algorithm is rigorously studied. It turns out that  the subsequential convergence to critical points with probability one is guaranteed. Although we only consider the same random vector inside both DC components for simplicity of presentation, the proposed algorithm and its convergence analysis can be extended to a more general setting which is $F(w) = \mathbb{E}(g(w,Z))-\mathbb{E}(h(w,\tilde{Z}))$, where $Z$ and $\tilde{Z}$ are two different random vectors. The extension aims to handle optimization problems involving with two parallel streams of data. Next, in the second and the third algorithms, we consider two scenarios where the values of $G$ and the subgradients of $H$ can be directly computed, respectively. The subsequential convergence to DC critical points is also established with milder assumptions than those of the first osDCA scheme. In three proposed algorithms, we require the number of samples used at each iteration to increase at a certain rate. This rate in the latter two algorithms is better than the first one. In addition, in the second scheme, this rate can be specified in advance without the knowledge on the complexity of a family of functions associated with $g$. The proposed osDCA schemes enjoy a double benefit of an  online algorithms: they are suitable to perform {\it streaming data} which come from an {\it unknown distribution}. Moreover, we discuss several contexts where one can formulate some classes of stochastic as well as deterministic programs into the form of (\ref{eq:ori}).

Finally, based on our proposed algorithms, we design two specific schemes for solving the expected problem of principal component analysis. Numerical experiments have been conducted carefully to study the proposed algorithms' behaviors in different aspects.

\section{Preliminaries}
\subsection{Outline of DC programming and DCA}
In this subsection, we briefly introduce DC programming and DCA. Let $\Gamma_0(\mathbb{R}^n)$ denote the convex cone of all lower semicontinuous proper convex functions on $\mathbb{R}^n$. The standard DC program takes the form
\begin{equation*}
\alpha := \inf\{f(x)= g(x) - h(x): x \in \mathbb{R}^n\} \quad (P_{dc}),
\end{equation*}
where $g,h \in \Gamma_0(\mathbb{R}^n)$. Such a function $f$ is called DC, $g-h$ is DC decomposition, while $g$ and $h$ are DC components of $f$. Note that, a DC program with closed convex constraint $x \in C$ can be equivalently written as a standard DC program in such a way that $f = (g+\chi_C) - h$, where $\chi_C$ is the indicator function of $C$.

For a convex function $\theta$ defined on $\mathbb{R}^n$ and a convex set $C$, the modulus of strong convexity of $\theta$ on $C$, denoted by $\rho(\theta,C)$ or $\rho(\theta)$ if $C = \mathbb{R}^n$, is given by
\begin{equation*}
\rho(\theta, C) = \sup\{\mu \geq 0: \theta - (\mu/2) \Vert \cdot \Vert^2 \text{ is convex on }C\}.
\end{equation*}
Moreover, a function $\theta$ is said to be strongly convex on $C$ if $\rho(\theta,C) >0$. The subdifferential of $\theta$ at $x_0 \in \dom \theta$, denoted by $\partial \theta(x_0)$, is defined by
\begin{equation*}
\partial \theta (x_0)=\{y \in \mathbb{R}^n: \theta(x) \geq \theta(x_0) + \langle x -x_0,y \rangle, \forall x \in \mathbb{R}^n\}.
\end{equation*}
The conjugate function $\theta^{\ast}$ of $\theta \in \Gamma_0(\mathbb{R}^n)$ is defined by
$\theta^{\ast}(y)=\sup\{\langle x,y \rangle-\theta(x) : x \in \mathbb{R}^n\}.$

A point $x^*$ is called a critical point, or a generalized Karush-Kuhn-Tucker (KKT) point of $(P_{dc})$ if $\partial g(x^*) \cap \partial h(x^*) \neq \emptyset$, or equivalently $0 \in \partial g(x^*) - \partial h(x^*)$, while it is called a strongly critical point of $g-h$ if $\emptyset \neq \partial h(x^*) \subset \partial g(x^*)$.

DCA is based on local optimality conditions and duality in DC programming,
which introduces the nice and elegant concept of approximating a DC program
by a sequence of convex ones: at each iteration $k$, DCA
approximates the second DC component $h$ by its affine minorization
$h_k(x) = h(x^k) + \langle x- x^k, y^k\rangle$, with $y^k \in \partial h(x^k)$,
and then solves the resulting convex subprogram  to get $x^{k+1}.$ The standard DCA is formally described as follows.

\bigskip
\noindent \textbf{Standard DCA.}

\noindent \textbf{Initialization:} Let $x^0 \in \dom \partial h$ and $k=0$.

\noindent \textbf{repeat}

Step 1: Compute the subgradient $y^{k}\in \partial h(x^{k})$.

Step 2: Solve the following convex program
\begin{equation*}
x^{k+1}\in \argmin \{g(x)-h_{k}(x):x\in X\}.
\end{equation*}

Step 3: $k = k+1$.

\noindent \textbf{until} Stopping criterion.
\bigskip

Convergences properties of the standard DCA and its complete theoretical foundation in the DC programming framework   can be
found in \cite{lethi05,phamdinh97,PLT98}. For instance, it is especially worth  mentioning   that
the sequence $\{x^k\}$ generated by DCA has the following properties:
\begin{itemize}
\item[1.] The sequence $\{(g-h)(x^k)\}$ is decreasing.
\item[2.] If $(g-h)(x^{k+1})=(g-h)(x^k)$, then $x^k$ and $x^{k+1}$ are critical points of $(P_{dc})$ and DCA terminates at $k$-th iteration.
\item[3.] If $\rho(g)+\rho(h)>0$ then the series $\sum_{k=1}^{\infty}{\Vert x^{k+1}-x^k \Vert^2}$ converges.
\item[4.] If the optimal value $\alpha$ of the problem $(P_{dc})$ is finite and the sequences $\{x^k\}$ and $\{y^k\}$ are bounded, then every limit point $\tilde{x}$ of $\{x^k\}$ is a critical point of $g-h.$
\end{itemize}

\subsection{Some notions in probability theory}
\subsubsection{History of a stochastic process}
Given a stochastic process $
\mathcal{X}=\{X_k\}_{k=1}^{\infty}$, we define the history up to time $k$ of $\mathcal{X}$ by
$\mathcal{P}_k = \sigma(X_1,X_2,\ldots,X_k),$ where $\sigma(X_1,X_2,\ldots,X_k)$ is the sigma algebra generated by random variables $\{X_1,X_2,\ldots,X_k\}$. The sequence of increasing sigma algebras $\{\mathcal{P}_k\}$ is called a filtration.
\subsubsection{Rademacher average}
For a set of points $\{z_1,z_2, \ldots,z_l\}:=z^l$ in $\Xi$, the Rademacher average $R_l(g,z^l)$ is defined as
\begin{equation*}
R_l(g,z^l) = \mathbb{E}_{\sigma} \sup_{w \in S} \left \vert \dfrac{1}{l} \sum_{i=1}^{l}{\sigma_i g(w,z_i)} \right \vert,
\end{equation*}
where $\sigma_i's$ are i.i.d. random numbers such that $\sigma_i \in \{\pm 1\}$ with $\mathbb{P}(\sigma_i = 1) = \mathbb{P}(\sigma_i =-1) = 1/2$. The Rademacher average of a family of functions $\{g(\cdot,z) : z \in \Xi\}$, denoted by $R_l(g,\Xi)$, is defined as
\begin{align*}
R_l(g,\Xi) = \sup_{z_1 \in \Xi, z_2 \in \Xi, \ldots,z_l \in \Xi} R_l(g,z^l).
\end{align*}

\subsection{Online Stochastic DCA for solving (\ref{eq:ori})}
\label{sec:osDCA}
This subsection develops osDCA schemes for solving the problem (\ref{eq:ori}) which can be described as follows.
\subsubsection{Problem setting}
 Let $P_Z$ be the probability distribution of $Z$ on $\mathbb{R}^n$ and $\Xi = \supp(P_Z)$ be the support of $P_Z$. By definition, a point $x \in \mathbb{R}^n$ is in $\supp(P_Z)$ if $P_Z(N_x) >0$, for all neighborhood $N_x$ of $x$. Since a measure ``lives" in its support, we only need to work in $\Xi$ instead of $\mathbb{R}^n$. For instance, a Dirac measure $\delta_a$ concentrating at a single point $a$ admits a support containing only one point $a$; a discrete measure $\mu = \sum_{i=1}^{\infty}{\beta_i \delta_{a_i}}$ with $\beta_i >0$ admits a support $\{a_1,a_2,\ldots\}$. A basic property of $\Xi$ is that it is closed in $\mathbb{R}^n$. Moreover, $P_Z(\Xi^\complement) = 0$ since $\mathbb{R}^n$ is the topological Hausdoff space and $P_Z$ is a Radon measure in $\mathbb{R}^n$. Therefore, only the values of $g$ and $h$ on $S \times \Xi$ matter. For simplicity of presentation, we assume that $\dom g = \dom h = S \times \Xi$. That is, the value of $g$ and $h$ outside $S \times \Xi$ is set to $+\infty$. Here we use the convention $+\infty - (+\infty) = +\infty$. Moreover, $g$ and $h$ are assumed to be Borel measurable. It is noted that the Borel sigma algebra on $\mathbb{R} \cup \{+\infty\}$ is generated by the order topology of $\mathbb{R} \cup \{+\infty\}$. We assume that $g(w,Z), h(w,Z)$ are integrable for all $w \in S$. Let $G(w) = \mathbb{E}(g(w,Z))$ and $H(w) = \mathbb{E}(h(w,Z))$, it follows that $\dom G = \dom H = S$. Besides, we assume that $g(\cdot,z)$ and $h(\cdot,z)$ are convex, lower semicontinuous, for all $z \in \Xi$, $G$ and $H$ are lower semicontinuous, so that the problem (\ref{eq:ori}) is DC. Moreover, we need some mild additional assumptions as follows.

\begin{assumption}\label{asp1}
\begin{itemize}
\item[i.] For all $z \in \Xi$, $\dom \partial h(\cdot,z) = S.$

\item[ii.] $\bar{\rho} :=\rho_H + \inf_{z \in \Xi} \rho(g(\cdot,z))>0.$

\item[iii.] There exists a Borel measurable selector $\tau$ such that
\begin{equation*}
\forall w \in S, z \in \Xi, \tau(w,z) \in \partial_w h(w,z),
\end{equation*}
where $\tau$ is $L^2$ uniformly bounded in the sense that there exists a Borel measurable function $\tilde{\tau}$ such that $\tilde{\tau}(Z)^2$ is integrable and $\forall w \in S, z \in \Xi, \Vert \tau(w,z) \Vert \leq \tilde{\tau}(z).$

\item[iv.] $\sup_{w \in S} \vert F(w) \vert <+\infty$.
\end{itemize}
\end{assumption}

\begin{remark} 
\label{rem1}
 It is observed that the assumptions i), iii) and iv) are mild. On another hand, thanks to the regularization technique introduced in \cite{phamdinh97}, the assumption \ref{asp1}-(ii) is easily fulfilled by adding an $L_2$ regularizer to both DC components.
\end{remark}

\begin{assumption}
\label{asp2}
\begin{itemize}
\item[i.] There exists a Borel measurable function $\tilde{g}: \mathbb{R}^n \to \mathbb{R}$ such that $\tilde{g}(Z)$ is integrable and 
\begin{equation*}
\vert g(w,z) \vert \leq \tilde{g}(z),~\forall w \in S, z \in \Xi.
\end{equation*}

\item[ii.] $R_k(g,\Xi) \leq  N_g/k^{\alpha}$ with $N_g>0$ and $\alpha >0.$
\end{itemize}
\end{assumption}

It is noteworthy that the assumption \ref{asp2}-(ii) holds for various cases described as follows \cite{ermoliev13}.

\textbf{Case 1. Holder functions }$g(\cdot,z), z \in \Xi$.

Let $D$ be the length of a cube in $\mathbb{R}^m$ containing the compact set $S$. Suppose that $ \exists M,L>0$ and $\gamma \in (0,1]$ such that
\begin{itemize}
\item[1.] $\vert g(w,z) \vert \leq M, \forall w \in S, z \in \Xi$,

\item[2.] $\vert g(x,z) - g(y,z) \vert \leq L \Vert x - y \Vert^{\gamma}, \forall x,y \in S, z \in \Xi.$ 
\end{itemize}
Then, for any $\alpha \in (0,1/2)$, $R_k(g,\Xi) \leq N_g/k^{\alpha}$, where
\begin{equation*}
N_g = L D^{\gamma} m^{\frac{\gamma}{2}} + \dfrac{M \sqrt{m}}{\sqrt{\gamma (1-2 \alpha) e}}.
\end{equation*}

\textbf{Case 2. Holder functions} $g(w, \cdot), w \in S$.

Suppose that $\Xi$ is compact, let $D$ be the length of a cube in $\mathbb{R}^n$ that contains $\Xi$. Suppose that there exists $M, L,\gamma >0$ such that
\begin{itemize}
\item[1.] $\vert g(w,z) \vert \leq M, \forall w \in S, z \in \Xi.$

\item[2.] $\vert g(w,u) - g(w,v) \vert \leq L \Vert u-v \Vert^{\gamma}, \forall w \in S, u,v \in \Xi$.
\end{itemize}
Then $R_k(g,\Xi) \leq N_g/k^{\alpha}$ where $N_g = M + L D^{\gamma} n^{\frac{\gamma}{2}}$ and $\alpha = \gamma/(2\gamma + n).$

\textbf{Case 3. Discrete set} $\Xi$.

Suppose that the number of elements of $\Xi$ is finite, say $\vert \Xi \vert = N_{\Xi}$. Furthermore, assume that there exists $M>0$ such that
$\vert g(w,z)\vert \leq M, ~\forall w \in S, z \in \Xi.$ Then, $R_k(g,\Xi) \leq M\sqrt{N_{\Xi}/k},$ hence, $\alpha = 1/2.$

It turns out that assumption \ref{asp2} is not strong; hence a class of functions meeting the criteria is wide to cover many problems arising in practice. In three cases of Rademacher complexity presented above, though $\alpha$ in case 2 can be very small in the high-dimension regime, which makes our next algorithm impractical, the other two cases have $\alpha =1/2$ or arbitrarily near to $1/2$, which are appropriate sample rates in practice.

It should be stressed that the Rademacher complexity measures the richness of a class of functions. Therefore, roughly speaking, the function $g$ must be quite ``simple" in this Rademacher sense. This criterion naturally fulfills our demand as we want to control the variability of stochastic approximations made on $g$. 

\subsubsection{Online Stochastic DCA schemes}
We now introduce an osDCA scheme described in algorithm \ref{alg:osDCA2}.

\begin{algorithm}
	\caption{Online Stochastic DCA}
	\label{alg:osDCA2}
\begin{algorithmic}
\State \textbf{Initialization:} Choose $w^0 \in S$ and a sequence of sample sizes $\{n_k\}$, set $k=0$.
\Repeat
	\State 1. Draw independently $n_k$ samples $Z_{k,1},\ldots,Z_{k,n_k}$ from the distribution of $Z$ in such a way that they are also independent of the past.
	
	\State 2. Compute $t^k = \dfrac{1}{n_k} \sum_{i=1}^{n_k}{\tau(w^k,Z_{k,i})}$.
	
	\State 3. Solve the following convex program to get $w^{k+1}$,
	\begin{align*}
	w^{k+1} \in \argmin_{w \in \mathbb{R}^m}\left\{\dfrac{1}{n_k} \sum_{i=1}^{n_k}{g(w,Z_{k,i})} -\langle t^k,w \rangle \right\}.
	\end{align*}
	
	\State 4. Set $k=k+1$.
\Until{Stopping criterion.}
\end{algorithmic}
\end{algorithm}
The algorithm \ref{alg:osDCA2} is well defined with probability $1$. To be more specific, the set of events that makes algorithm \ref{alg:osDCA2} work is $\mathcal{V} = \cap_{k=1}^{\infty} \cap_{i=1}^{n_k} (Z_{k,i} \in \Xi)$ and hence $\mathbb{P}(\mathcal{V})=1.$ We denote $Z_k = Z_{k,1:n_k}$ and $\mathcal{P}_k = \sigma(Z_0,Z_1,\ldots,Z_{k-1},w^0,w^1,\ldots,w^k)$. We observe that $\{w^k\}_{k=0}^{\infty}$ is a predictable process and $\{t^k\}_{k=0}^{\infty}$ is an adapted process with respect to the filtration $\{\mathcal{P}_{k+1}\}_{k=0}^{\infty}.$ The convergence results of algorithm \ref{alg:osDCA2} are presented in theorem \ref{theoremosDCA2}.

\begin{theorem}
\label{theoremosDCA2}
Under assumptions \ref{asp1} and \ref{asp2}, let $\beta = \min\{\alpha,1\}$, if the sequence of sample sizes $\{n_k\}$ satisfies $\sum_{k=1}^{\infty}{n_k^{-\beta}} <+\infty$, the iterations of algorithm \ref{alg:osDCA2} satisfy:

1. There exists $F^{\infty}$ integrable such that $F(w^k) \to F^{\infty}$ a.s.

2. $\sum_{k=1}^{\infty}\Vert w^{k+1} - w^k \Vert^2 < +\infty$ a.s.

3. There exists a measurable set $\mathcal{L} \subset \Omega$ with $\mathbb{P}(\mathcal{L}) = 1$ such that for each $\omega \in \mathcal{L}$, every limit point of $\{w^k(\omega)\}$ is a critical point of $F = G-H.$
\end{theorem}
\begin{proof}
1. Let $\nu(w):= \mathbb{E}(\tau(w,Z))$. It follows from $\nu(w^k) \in \partial H(w^k)$ that
\begin{equation}
\label{eqpls:01}
H(w^{k+1}) \geq H(w^k) + \langle \nu(w^k), w^{k+1}-w^k\rangle + \dfrac{\rho_H}{2}\Vert w^{k+1}-w^k \Vert^2.
\end{equation}
On the other hand, it follows from definition of $w^{k+1}$ that
\begin{align}
\label{eqpls:02}
\dfrac{1}{n_k} \sum_{i=1}^{n_k}{g(w^{k},Z_{k,i})} \geq &\dfrac{1}{n_k} \sum_{i=1}^{n_k}{g(w^{k+1},Z_{k,i})} + \langle t^k, w^k-w^{k+1}\rangle +\dfrac{\inf_{z\in \Xi}\rho(g(\cdot,z))}{2} \Vert w^{k+1} -w^k \Vert^2.
\end{align}
From (\ref{eqpls:01}) and (\ref{eqpls:02}), we obtain

\begin{equation}\label{eq:07}
\begin{array}{ll}
\dfrac{1}{n_k} \sum_{i=1}^{n_k} g(w^{k+1},Z_{k,i}) \leq &H(w^{k+1}) + \dfrac{1}{n_k} \sum_{i=1}^{n_k} g(w^k,Z_{k,i}) \\
&- H(w^k)-\dfrac{\bar{\rho}}{2} \Vert w^{k+1} - w^k \Vert^2 + \langle t^k -\nu(w^k), w^{k+1}-w^k\rangle,
\end{array}
\end{equation}
with $\bar{\rho} = \rho_H + \inf_{z\in \Xi}\rho(g(\cdot,z)).$ By taking conditional expectation with respect to $\mathcal{P}_k$ both sides of (\ref{eq:07}), we obtain
\begin{align}
\label{eqpls:03}
&\mathbb{E}(F(w^{k+1}) - F(w^k) \vert \mathcal{P}_k) \leq  \mathbb{E}(\langle t^k - \nu(w^k),w^{k+1}-w^k\rangle \vert \mathcal{P}_k) \notag \\ &+\mathbb{E}\left( G(w^{k+1}) - \dfrac{1}{n_k} \sum_{i=1}^{n_k} g(w^{k+1},Z_{k,i}) \vert \mathcal{P}_k \right)-\dfrac{\bar{\rho}}{2}\mathbb{E}(\Vert w^{k+1}-w^k \Vert^2 \vert \mathcal{P}_k).
\end{align}

By applying Schwartz inequality and Holder inequality,
\begin{align}
\label{eqpls:04}
&\mathbb{E}(\langle t^k - \nu(w^k), w^{k+1} - w^k \rangle \vert \mathcal{P}_k) \leq \mathbb{E}(\Vert t^k - \nu(w^k) \Vert^2 \vert \mathcal{P}_k)^{\frac{1}{2}} \mathbb{E}(\Vert w^{k+1}-w^k \Vert^2 \vert \mathcal{P}_k)^{\frac{1}{2}}.
\end{align}
By using AM-GM inequality, we obtain
\begin{align}
\label{eqpls:05}
&\mathbb{E}(\Vert t^k - \nu(w^k) \Vert^2 \vert \mathcal{P}_k)^{\frac{1}{2}} \mathbb{E}(\Vert w^{k+1}-w^k \Vert^2 \vert \mathcal{P}_k)^{\frac{1}{2}} \leq \dfrac{1}{2\bar{\rho}}\mathbb{E}(\Vert t^k - \nu(w^k) \Vert^2 \vert \mathcal{P}_k) + \dfrac{\bar{\rho}}{2}\mathbb{E}(\Vert w^{k+1}-w^k \Vert^2 \vert \mathcal{P}_k).
\end{align}
It follows from the independence of $Z_{k,i}$ and $Z_{k,j}$ for all $i \neq j$ that
\begin{align*}
\mathbb{E} \left( \Vert t^k -\nu(w^k) \Vert^2 \vert \mathcal{P}_k \right) &= \dfrac{1}{n_k^2} \sum_{i=1}^{n_k} \mathbb{E} \left( \Vert \tau(w^k,Z_{k,i}) -\nu(w^k) \Vert^2 \vert  \mathcal{P}_k \right).
\end{align*}
We observe that
\begin{align*}
&\mathbb{E}\left(\Vert \tau(w^k,Z_{k,i}) -\nu(w^k) \Vert^2 \vert \mathcal{P}_k \right) \\
&= \mathbb{E}_Z(\Vert \tau(w^k,Z) \Vert^2) + \Vert \nu(w^k)\Vert^2 -2 \langle \mathbb{E}(\tau(w^k,Z_{k,i}) \vert \mathcal{P}_k), \nu(w^k)\rangle\\
&= \mathbb{E}_Z \left( \Vert \tau(w^k,Z) \Vert^2 \right)-\Vert \nu(w^k) \Vert^2=\mathbb{V}_Z(\tau(w^k,Z)).
\end{align*}
Therefore,
\begin{equation}
\label{eq:06}
\mathbb{E} \left( \Vert t^k -\nu(w^k) \Vert^2 \vert \mathcal{P}_k \right) = \dfrac{1}{n_k} \mathbb{V}_Z(\tau(w^k,Z)).
\end{equation}

From (\ref{eqpls:03}), (\ref{eqpls:04}), (\ref{eqpls:05}), and (\ref{eq:06}) we obtain
\begin{align}
\label{eqnewwww}
&\mathbb{E}(F(w^{k+1}) - F(w^k) \vert \mathcal{P}_k) \leq \dfrac{\mathbb{V}_Z(\tau(w^k,Z))}{2 \bar{\rho} \times n_k} +\mathbb{E}\left( G(w^{k+1}) - \dfrac{1}{n_k} \sum_{i=1}^{n_k} g(w^{k+1},Z_{k,i}) \vert \mathcal{P}_k \right).
\end{align}

Next, we make an upper bound on the right-hand side of (\ref{eqnewwww}). Firstly, the (nonnegative) term $\mathbb{V}_Z(\tau(w^k,Z))$ is bounded above by $\mathbb{E}(\tilde{\tau}(Z)^2).$ Secondly, we show that
\begin{equation}
\label{eq:rademacher}
\mathbb{E}\left( \sup_{w \in S} \left \vert G(w) - \dfrac{1}{n_k} \sum_{i=1}^{n_k} g(w,Z_{k,i}) \right \vert \right)\leq 2 R_{n_k}(g,\Xi).
\end{equation}
To prove (\ref{eq:rademacher}), let us first introduce ``ghost samples" $Z'_{k,1},Z'_{k,2},\ldots,Z'_{k,n_k}$ (similar to the arguments in \cite{boucheron05}) that are independent of all $Z_{k,i}$ and identically distributed with $Z$. By Jensen's inequality, we get
\begin{align*}
&\left \vert \dfrac{1}{n_k} \sum_{i=1}^{n_k} g(w,Z_{k,i}) -\mathbb{E}(g(w,Z)) \right \vert\leq \mathbb{E} \left( \left \vert \dfrac{1}{n_k} \sum_{i=1}^{n_k} {\left(g(w,Z_{k,i})-g(w,Z'_{k,i})\right)} \right \vert \vert Z_{k,i},i=\overline{1,n_k} \right).
\end{align*}
Therefore,
\begin{align*}
&\mathbb{E} \left( \sup_{w \in S} \left \vert \dfrac{1}{n_k} \sum_{i=1}^{n_k} g(w,Z_{k,i}) -\mathbb{E}(g(w,Z)) \right \vert \right) \leq \mathbb{E} \left( \sup_{w \in S} \left \vert \dfrac{1}{n_k} \sum_{i=1}^{n_k} g(w,Z_{k,i}) - \dfrac{1}{n_k} \sum_{i=1}^{n_k} g(w,Z'_{k,i}) \right \vert \right).
\end{align*}
Now let $\sigma_1, \sigma_2,\ldots, \sigma_{n_k}$ be independent random variables with $\mathbb{P}(\sigma_i = 1) = \mathbb{P}(\sigma_i = -1) = \frac{1}{2}$ in such a way that they are also independent of $Z_{k,i}$ and $Z'_{k,i}$. Then,
\begin{align*}
&\mathbb{E} \left( \sup_{w\in S} \left \vert \dfrac{1}{n_k} \sum_{i=1}^{n_k} g(w,Z_{k,i}) - \dfrac{1}{n_k} \sum_{i=1}^{n_k} g(w,Z'_{k,i}) \right \vert \right) \\
&= \mathbb{E} \left( \sup_{w \in S} \left \vert \dfrac{1}{n_k} \sum_{i=1}^{n_k} \sigma_i (g(w,Z_{k,i})-g(w,Z'_{k,i})) \right \vert \right)\\
&\leq 2 \mathbb{E} \left( \sup_{w \in S} \dfrac{1}{n_k} \left \vert \sum_{i=1}^{n_k} \sigma_i g(w,Z_{k,i}) \right \vert \right)\\
&= 2\mathbb{E} \left( \mathbb{E}_{\sigma} \left( \sup_{w \in S} \dfrac{1}{n_k} \left \vert \sum_{i=1}^{n_k} \sigma_i g(w,Z_{k,i}) \right \vert \right) \right)\\
&=2\mathbb{E}(R_{n_k}(g,Z^{n_k})) \leq 2\mathbb{E}(R_{n_k}(g,\Xi)) = 2R_{n_k}(g,\Xi).
\end{align*}

Now, we establish the almost sure convergence of the sequence $\{F(w^k)\}$ as follows. The assumption \ref{asp1}-(iv) implies that there exists $R$ such that $F(w) \geq R, \forall w \in S.$ Let $D(w) = F(w)-R \geq 0$ and $S_k = [\mathbb{E}(D(w^{k+1}) -D(w^k) \vert \mathcal{P}_k)>0]$. Since $S_k$ is $\mathcal{P}_k$-measurable and by using (\ref{eqnewwww}), (\ref{eq:rademacher}) , we obtain
\begin{align*}
&\sum_{k=1}^{\infty}\mathbb{E}(1_{S_k}(D(w^{k+1}) - D(w^k))) \\
&=\sum_{k=1}^{\infty}{\mathbb{E}\left( \mathbb{E}(1_{S_k}(D(w^{k+1})-D(w^k)) \vert \mathcal{P}_k) \right)} \\
&\leq \dfrac{1}{2 \bar{\rho}} \sum_{k=1}^{\infty} \dfrac{\mathbb{E}(\mathbb{V}_Z(\tau(w^k,Z)))}{n_k} + 2 \sum_{k=1}^{\infty} R_{n_k}(g,\Xi) \\
&\leq \dfrac{\mathbb{E}(\tilde{\tau}(Z)^2)}{2 \bar{\rho}} \sum_{k=1}^{\infty}{\dfrac{1}{n_k}} + 2 N_g\sum_{k=1}^{\infty} \dfrac{1}{n_k^{\alpha}} <+\infty.
\end{align*}
It follows from semimartingale convergence theorem \cite{metivier82} that there exists $D^{\infty}$ integrable such that $D(w^k) \to D^{\infty}$ a.s., which implies $F(w^k) \to F^{\infty}=D^{\infty}+R$ a.s.

2. By applying AM-GM inequality, we obtain
\begin{align*}
\langle t^k - \nu(w^k), w^{k+1} -w^k \rangle \leq &\dfrac{1}{\bar{\rho}} \Vert t^k -\nu(w^k) \Vert^2 + \dfrac{\bar{\rho}}{4} \Vert w^{k+1} -w^k \Vert^2.
\end{align*}
Combining this inequality with (\ref{eq:07}), we get
\begin{align*}
&\dfrac{\bar{\rho}}{4} \Vert w^{k+1} -w^k \Vert^2 \leq F(w^k) - F(w^{k+1}) +G(w^{k+1}) -G(w^{k})\\&-\dfrac{1}{n_k} \sum_{i=1}^{n_k} g(w^{k+1},Z_{k,i})+\dfrac{1}{n_k} \sum_{i=1}^{n_k} g(w^{k},Z_{k,i}) +\dfrac{1}{\bar{\rho}}\Vert t^k -\nu(w^k) \Vert^2.
\end{align*}
By applying Lebesgue dominated convergence theorem (theorem 4.2, \cite{brezis10}) and noticing that
\begin{align*}
\mathbb{E}\left( \mathbb{E} \left(\dfrac{1}{n_k} \sum_{i=1}^{n_k} g(w^{k},Z_{k,i}) - G(w^k) \vert \mathcal{P}_k \right) \right)=0,
\end{align*}
we get
\begin{align*}
&\dfrac{\bar{\rho}}{4} \mathbb{E}\left( \sum_{k=1}^{\infty} \Vert w^k-w^{k+1} \Vert^2 \right) \leq \mathbb{E}(F(w^1)) - \mathbb{E}(F^{\infty})\\&+\dfrac{M}{\bar{\rho}} \sum_{k=1}^{\infty} \dfrac{1}{n_k} + 2 N_g\sum_{k=1}^{\infty} \dfrac{1}{n_k^{\alpha}} < \infty.
\end{align*}
Therefore, $\sum_{k=1}^{\infty} \Vert w^k-w^{k+1} \Vert^2 < +\infty \text { a.s.}$

3. We denote $G_k(w) = \frac{1}{n_k} \sum_{i=1}^{n_k} g(w,Z_{k,i})$, it follows from $t^k \in \partial G_k(w^{k+1})$ and $\nu(w^k) \in \partial H(w^k)$ that
\begin{align*}
&\langle w^{k+1},t^k \rangle = G_k(w^{k+1}) + G^{\ast}_k(t^k), \\
&\langle \nu(w^k),w^k \rangle = H(w^k) + H^{\ast}(\nu(w^k)).
\end{align*}
Together with the following inequalities
\begin{align*}
&H(w^{k+1}) \geq H(w^k) + \langle \nu(w^k),w^{k+1} - w^k \rangle,\\
&G_k(w^{k+1}) -\langle t^k, w^{k+1} \rangle \leq G_k(w^k) - \langle t^k,w^k \rangle,
\end{align*}
we obtain
\begin{align*}
&G_k(w^k) - H(w^k) \geq H^{\ast}(\nu(w^k)) - G_k^{\ast}(t^k) + \langle t^k -\nu(w^k),w^k \rangle\\
&\geq G_k(w^{k+1})-H(w^{k+1}) +\langle t^k -\nu(w^k),w^k-w^{k+1} \rangle,
\end{align*}
which implies
\begin{equation}
\label{eq:08}
G_k(w^k) - H(w^k) - H^{\ast}(\nu(w^k)) + G_k^{\ast}(t^k) \to 0,
\end{equation}
since $G_k(w^k)-H(w^k) \to F^{\infty}$, $t^k - \nu(w^k) \to 0$, and $ G_k(w^{k+1}) - H(w^{k+1}) \to F^{\infty}$.

Hence, (\ref{eq:08}) implies $G(w^k) + G_k^{\ast}(t^k)- \langle w^k,\nu(w^k) \rangle \to 0.$

It is observed that
\begin{align*}
&\vert G_k^{\ast}(t^k) - G^{\ast}(t^k) \vert = \left \vert \sup_{x \in S} \{\langle x,t^k \rangle - G_k(x)\} - \sup_{x \in S}\{\langle x,t^k \rangle - G(x)\} \right \vert \leq \sup_{x \in S}\vert G_k(x) - G(x) \vert \to 0.
\end{align*}
Hence, we obtain $G(w^k) + G^{\ast}(t^k) - \langle w^k,\nu(w^k) \rangle \to 0$ a.s. Now let $\mathcal{L}$ be an intersection of sets with probability $1$ gained from all almost surely true statements from the beginning of the proof, we have $\mathbb{P}(\mathcal{L})=1$ since there are at most countably finite statements. Let $\omega \in \mathcal{L}$, we have $\{w^k(\omega)\}$ and $\{\nu(w^k(\omega))\}$ are bounded. Let $w^{\ast} \in S$ be a limit point of $\{w^k(\omega)\}$, there exists a subsequence $\{w^{k_j}(\omega)\}$ such that $w^{k_j}(\omega) \to w^{\ast}$. By extracting a subsequence of $\{\nu(w^{k_j}(\omega))\}$ if necessary, we can assume that $\nu(w^{k_j}(\omega)) \to \nu^{\ast}$, which implies $t^{k_j}(\omega) \to \nu^{\ast}$. Therefore, $G(w^{k_j}(\omega)) + G^{\ast}(t^{k_j}(\omega)) \to \langle w^{\ast}, \nu^{\ast} \rangle.$ By letting $j \to +\infty$ and noting that $\theta(w,z) = G(w) +G^{\ast}(z)$ is lower semicontinuous, we obtain $G(w^{\ast})+G^{\ast}(\nu^{\ast}) \leq \langle w^{\ast},\nu^{\ast} \rangle$. On the other hand, according to Young's inequality, $G(w^{\ast})+G^{\ast}(\nu^{\ast}) \geq \langle w^{\ast},\nu^{\ast} \rangle$. Thefore, $G(w^{\ast})+G^{\ast}(\nu^{\ast}) = \langle w^{\ast},\nu^{\ast} \rangle$. In other words, $\nu^{\ast} \in \partial G(w^{\ast})$. Furthermore, for each $w \in S$, it follows from $\nu(w^{k_j}(\omega)) \in \partial H(w^{k_j}(\omega))$ that $H(w) \geq H(w^{k_j}(\omega))+ \langle \nu(w^{k_j}(\omega)),w-w^{k_j}(\omega) \rangle,$ which implies $H(w) \geq H(w^{\ast}) + \langle \nu^{\ast}, w- w^{\ast} \rangle.$ Therefore, $\nu^{\ast} \in \partial H(w^{\ast})$, and we conclude that $w^{\ast}$ is a critical point of $F=G-H$ since $\partial G(w^{\ast}) \cap \partial H(w^{\ast}) \neq \emptyset.$
\end{proof}

\begin{remark}
\label{rem2} 
(i) The algorithm only uses samples at the current time to update the solution (past samples are no longer used). Therefore, even if the distribution of $Z$ changes at a certain time (suppose that, due to some real-world events, $Z$ becomes $Z'$ at the iteration $k$), the algorithm will automatically solve the problem (\ref{eq:ori}) with $Z$ being replaced by $Z'$. Indeed, the current solution $w^k$ can be considered as the initial point for restart, the algorithm continues operating based on new samples from the distribution of $Z'$. The theorem \ref{theoremosDCA2} is still valid, and the subsequential convergence with probability one to DC critical points of the DC problem associated with the new distribution is guaranteed. This is indeed an advantage of the osDCA. In contrast, intuitively, stochastic algorithms using aggregated update (still using old samples to compute the current solution) barely have this kind of adaptivity. We will conduct numerical experiments to study this aspect.
\\
(ii) Our algorithm and the convergence analysis can be extended to deal with the more general problem whose the random variables inside the first and the second DC components are not necessarily the same, i.e.,
$F(w) = \mathbb{E}(g(w,Z))-\mathbb{E}(h(w,\tilde{Z}))$. With this new setting, at the iteration $k$, we approximate values of $G$ and the subgradients of $H$ by using $n_k$ independent random samples obtained from the distribution of $Z$ and $\tilde{n}_k$ independent random samples obtained from the distribution of $\tilde{Z}$, respectively. The sample size sequences $\{n_k\}$ and $\{\tilde{n}_k\}$  need to increase in such a way that $\sum_{k=1}^{\infty}{n_k^{-\alpha}}<\infty$ and $\sum_{k=1}^{\infty}{\tilde{n}_k^{-1}}<\infty$.
\end{remark}

Next, we will discuss two scenarios where one can directly compute (without stochastically approximation) values of $G$ or subgradients of $H$. Since the information of $G$ (resp. subgradient of $H$) can be achieved, we will modify the algorithm \ref{alg:osDCA2} to exploit this advantage. Note that these two schemes are not special cases of the algorithm \ref{alg:osDCA2}, but they will coincide with the algorithm \ref{alg:osDCA2} in some cases.

\paragraph*{The values of $G$ can be directly computed without approximation}

In this case, $G$ does not need to be stochastically approximated, we replace the approximation of $G$ in step 3 of algorithm \ref{alg:osDCA2} by its true value, which results in algorithm \ref{alg:osDCA_noiselessG}.

\begin{algorithm}
	\caption{Online Stochastic DCA with exact $G$}
	\label{alg:osDCA_noiselessG}
Similar to algorithm \ref{alg:osDCA2}, where step 3 of algorithm \ref{alg:osDCA2} is replaced by the following step:
\begin{algorithmic}

	\State 3. Solve the following convex program to get $w^{k+1}$,
	\begin{align*}
	w^{k+1} \in \argmin_{w \in \mathbb{R}^m}\left\{G(w) -\langle t^k,w \rangle \right\}.
	\end{align*}

\end{algorithmic}
\end{algorithm}

With this algorithm, we obtain stronger convergence results since $G$ is computed exactly. Note that, in the convergence results of algorithm \ref{alg:osDCA2}, we impose the assumption \ref{asp2} in order to control the variance of the stochastic estimator of $G$. To study the convergence of algorithm \ref{alg:osDCA_noiselessG}, we do not need such an assumption. Furthermore, in the assumption \ref{asp1}, we replace the convexity condition $\rho_H + \inf_{z\in \Xi} \rho(g(\cdot,z))>0$ by the weaker one $\rho_H + \rho_G>0,$ which gives rise to a milder assumption called the assumption 1'. We obtain the convergence theorem \ref{theoremos_pls} whose proof is similar to the proof of theorem \ref{theoremosDCA2}.

\begin{theorem}
\label{theoremos_pls}
Under the assumption 1', if the sequence of sample sizes $\{n_k\}$ satisfies $\sum_{k=1}^{\infty}{n_k^{-1}} <+\infty$, then the iterations of algorithm \ref{alg:osDCA_noiselessG} satisfy:

1. There exists $F^{\infty}$ integrable such that $F(w^k) \to F^{\infty}$ a.s.

2. $\sum_{k=1}^{\infty}{\Vert w^{k+1} - w^k \Vert^2} < +\infty$ a.s.

3. There exists a measurable set $\mathcal{L} \subset \Omega$ with $\mathbb{P}(\mathcal{L}) = 1$ such that for each $\omega \in \mathcal{L}$, every limit point of $\{w^k(\omega)\}$ is a critical point of $F = G-H.$
\end{theorem}

\paragraph*{The subgradients of $H$ can be directly computed without approximation}

In this case, we replace the stochastic estimator of the subgradient of $H$ in the algorithm \ref{alg:osDCA2} by the true subgradient of $H$ to obtain the following algorithm.

\begin{algorithm}
	\caption{Online Stochastic DCA with exact subgradients of $H$}
	\label{alg:osDCA_noiselessH}
Similar to algorithm \ref{alg:osDCA2}, where step 2 of algorithm \ref{alg:osDCA2} is replaced by the following step:
\begin{algorithmic}

	\State 2. Compute $t^k \in \partial H(w^k)$.

\end{algorithmic}
\end{algorithm}

Since we work directly on $H$, we replace assumption \ref{asp1}-(i) by $\dom \partial H = S.$ Likewise, the assumption \ref{asp1}-(iii) is replaced by the following:
\begin{align*}
\text{there exist } M>0 \text{ such that }
\forall w \in S, \forall t \in \partial H(w): \Vert t \Vert \leq M.
\end{align*}
These modifications bring about a new set of assumptions called assumption 1''. We obtain the following convergence results whose proof is similar to the proof of algorithm \ref{theoremosDCA2}.

\begin{theorem}
\label{theoremosDCA2_Gnoiseless}
Under assumptions 1'' and \ref{asp2}, if the sequence of sample sizes $\{n_k\}$ satisfies $\sum_{k=1}^{\infty}{n_k^{-\alpha}} <+\infty$, the iterations of algorithm \ref{alg:osDCA_noiselessH} satisfy:

1. There exists $F^{\infty}$ integrable such that $F(w^k) \to F^{\infty}$ a.s.

2. $\sum_{k=1}^{\infty}\Vert w^{k+1} - w^k \Vert^2 < +\infty$ a.s.

3. There exists a measurable set $\mathcal{L} \subset \Omega$ with $\mathbb{P}(\mathcal{L}) = 1$ such that for each $\omega \in \mathcal{L}$, every limit point of $\{w^k(\omega)\}$ is a critical point of $F = G-H.$
\end{theorem}

\begin{remark}
\label{rem3}
(i) When $g(w,z)$ does not depend on $z$, algorithm \ref{alg:osDCA_noiselessG} coincides with algorithm \ref{alg:osDCA2}; likewise, when $h(w,z)$ does not depend on $z$, algorithm \ref{alg:osDCA_noiselessH} and algorithm \ref{alg:osDCA2} coincide. It is worth noting that, in practice, thanks to the flexibility of DC decompositions, one can usually formulate the given stochastic problem as a stochastic DC program with one stochastic DC component and one deterministic DC component. For example, we consider $F(w) = \mathbb{E}(f(w,Z))$. If the functions $f(\cdot,z)$ are $L$-smooth with the same constant $L$ for all $z \in \Xi$. Then, $F$ has the following DC decomposition:
\begin{equation*}
F(w) = \underbrace{\dfrac{L}{2}\Vert w \Vert^2}_{G(w)} - \underbrace{\mathbb{E}\left(\dfrac{L}{2}\Vert w \Vert^2-f(w,Z) \right)}_{H(w)}.
\end{equation*}
In another case, suppose that there exists a convex function $\varphi(w)$ such that functions $f(w,z)+\varphi(w)$ are convex for all $z \in \Xi$ (in particular, when $\varphi(w) = (\kappa/2)\Vert w \Vert^2$, $f(\cdot,z)$ are weakly convex), $F$ has the following DC decomposition:
\begin{equation*}
F(w) = \underbrace{\mathbb{E}\left(f(w,Z)+\varphi(w) \right)}_{G(w)} - \underbrace{\varphi(w)}_{H(w)}.
\end{equation*} 
(ii) In big data analytics, large-sum problems play a key role. We consider the following large-sum objective function
\begin{equation*}
F(w) = \sum_{i=1}^{N}{\alpha_i f_i(w)} = \sum_{i=1}^{N}{\alpha_i g_i(w)} - \sum_{i=1}^{N}{\alpha_i h_i(w)},
\end{equation*}
where $g_i, h_i$ are convex, $\alpha_i \geq 0$ for all $i = \overline{1,N}$ and $\sum_{i=1}^{N}{\alpha_i} = 1.$
The function $F$ can be rewritten as $F(w) = \mathbb{E}(g_I(w)) - \mathbb{E}(h_I(w))$, where $I$ is a random index with $\mathbb{P}(I=i)=\alpha_i$. In this case, the distribution of $I$ is known completely. However, as $N$ can be very large, we may still need to apply osDCA schemes. Furthermore, since $I$ is known, we have full freedom to choose algorithm \ref{alg:osDCA2}, algorithm \ref{alg:osDCA_noiselessG}, or algorithm \ref{alg:osDCA_noiselessH} to apply, which leads to - in general - three distinctive algorithms. The practical trade-off between these algorithms would be which DC component (or none of them) is cheaper to be computed directly.
\end{remark}

\section{Applications: solving the Expected PCA}
Principal component analysis (PCA) is arguably one of the most successful tools for dimensionality reduction. In this section, we will apply osDCA schemes to the expected problem of PCA to study the generalization capacity of the proposed methods.

\subsection{osDCA schemes for solving Expected PCA}
We consider the following expected problem of PCA (denoted by E-PCA) as follows \cite{montanari15},

\begin{align*}
\quad \text{min}& \quad -\frac{1}{2}\mathbb{E}(\langle w,Z \rangle^2), \quad \text{subject to} \quad \Vert w \Vert \leq 1, \quad (\text{E-PCA})
\end{align*}
where $Z$ is a normalized random vector, i.e. $\Vert Z \Vert = 1$, with unknown distribution. The situation in which we are interested is that the data obtained online.

The problem (E-PCA) can be considered as the theoretical problem of the classic PCA (and - vice versa - the classic PCA is the empirical problem of (E-PCA)). In other words, the problem (E-PCA) aims to generalize the compressing capacity of the classical PCA on unseen data.

Firstly, we observe that the problem (E-PCA) is nonconvex and it can be formulated as a DC problem,
\begin{align}
\label{firstDC}
\quad \underset{w \in S}{\minimize} \quad G(w)-H(w),
\end{align}
where $G(w)=\frac{\lambda}{2}\Vert w \Vert^2, H(w) = \mathbb{E}\left( \frac{\lambda}{2} \Vert w \Vert^2 + \frac{1}{2} \langle w,Z \rangle^2 \right)$, $S = \{w \in \mathbb{R}^m: \Vert w \Vert \leq 1\}$ and $\lambda >0$. Although we have a very natural DC decomposition with $G(w)=0, H(w)=\mathbb{E}\left( \frac{1}{2} \langle w,Z \rangle^2 \right)$, here we add $\frac{\lambda}{2} \Vert w \Vert^2$ to both DC components to fulfill to assumption \ref{asp1}-(ii). Since the values $G$ are directly obtained without approximation, algorithm \ref{alg:osDCA2} coincides with algorithm \ref{alg:osDCA_noiselessG}. We call this scheme osDCA-1, where the $k$-th iteration is described as follows.

1. Receive $n_k$ samples $Z_{k,1},\ldots,Z_{k,n_k}$.
	
2. Compute $t^k = \lambda w^k + \frac{1}{n_k} \sum_{i=1}^{n_k}{\langle w^k,Z_{k,i} \rangle} Z_{k,i}$.
	
3. Update $w^{k+1} = \begin{cases} &\lambda^{-1} t^k \text{ if } \Vert t^k \Vert \leq \lambda\\
	 & t^k /\Vert t^k \Vert \text{ otherwise.}\end{cases}$

Secondly, it is well-known that if a function $\theta$ has $L$-Lipschitz continuous gradient, then $(L/2)\Vert \cdot \Vert^2 - \theta$ and $(L/2)\Vert \cdot \Vert^2 + \theta$ are convex. Therefore, we have another DC decomposition for the problem (E-PCA) as follows,
\begin{align}
\label{secondDC}
\underset{w \in S}{\minimize} \quad G(w)-H(w).
\end{align}
where
\begin{align*}
    G(w) &= \mathbb{E} \left(\dfrac{L}{2} \Vert w \Vert^2 - \dfrac{1}{2} \langle w,Z \rangle^2 \right),\\
    H(w) &= \mathbb{E} \left( \dfrac{L}{2} \Vert w \Vert^2 +\dfrac{1}{2} \langle w,Z \rangle^2 \right).
\end{align*}
Since $G, H$ remains unknown, we apply algorithm \ref{alg:osDCA2} for this DC problem. Obviously the family $\{g(\cdot,z) : \Vert z \Vert = 1\}$ is uniformly Lipschitz and uniformly bounded by a constant, therefore, the rate $\alpha$ in assumption \ref{asp2} can be chosen arbitrarily in $(0,1/2)$. With this setup, we obtain a second scheme called osDCA-2 whose the $k$-th iteration is described as follows.

1. Receive $n_k$ samples $Z_{k,1},Z_{k,2},\ldots,Z_{k,n_k}$.

2. Compute the stochastic gradient
\begin{align*}
t^k = L w^k + \dfrac{1}{n_k} \sum_{i=1}^{n_k}{\langle w^k,Z_{k,i} \rangle Z_{k,i}}.
\end{align*}

3. Solve the following convex program to get $w^{k+1}$,
\begin{align}
\label{cnvsub}
\underset{w \in S}{\minimize}\left\{ \dfrac{L}{2} \Vert w \Vert^2 - \dfrac{1}{2 n_k} \sum_{i=1}^{n_k}{\langle w,Z_{k,i} \rangle^2} - \langle t^k,w \rangle \right\}.
\end{align}

The problem (\ref{cnvsub}) is convex and can be solved by existing convex optimization packages. However, we solve it by DCA since it has the following ``false" DC decomposition
\begin{align*}
&\tilde{g}(w) = \dfrac{L}{2} \Vert w \Vert^2,\tilde{h}(w) = \dfrac{1}{2n_k} \sum_{i=1}^{n_k}{\langle w,Z_{k,i} \rangle^2 +\langle t^k,w \rangle,}
\end{align*}
which results in a simple DCA scheme where convex subproblems have closed-form solutions.  The (deterministic) DCA takes the current solution $u^0 = w^k$ as the initial point, then operates until the stopping criterion which is $\Vert u^{l+1} - u^{l} \Vert < \epsilon$ is met, where $\epsilon>0$ is the error tolerance.

\subsection{Numerical experiments}
\subsubsection{Datasets}
The numerical experiments are conducted on standard machine learning datasets on LIBSVM \footnote{The datasets can be downloaded from \url{https://www.csie.ntu.edu.tw/~cjlin/libsvm/}.}. The information of the used datasets is described in Table 1. The samples of each dataset are normalized as $\Vert z_i \Vert = 1.$

\begin{table}[H]
\label{datainfo}
\begin{center}
\begin{tabular}{>{\centering\arraybackslash}p{3.1cm}|>{\centering\arraybackslash}p{1.9cm}|>{\centering\arraybackslash}p{1.55cm}|>{\centering\arraybackslash}p{2.1cm}}
Dataset & \quad \# Features \quad & \quad \# Train set \quad & \quad \# Validation set \quad \\
\hline
letter & 16 & 15000 & 5000 \\
YearPredictionMSD & 90 & 463715 & 51630 \\
SensIT Vehicle & 100 & 78823 & 19705 \\
shuttle & 9 & 43500 & 14500 \\
\end{tabular}
\end{center}
\caption{Datasets' information}
\end{table}

Furthermore, to test the adaptive ability of osDCA schemes, we generate a synthetic dataset that consists of two subdatasets (training set ($200000 \times 500$), validation set ($500000 \times 500$)) and (training set ($200000 \times 500$), validation set ($200000 \times 500$)), in which the generating mechanism is described in subsection \ref{exp_res}.
\subsubsection{Comparative algorithms}

We compare our algorithms with two versions of Projected Stochastic Subgradient method (PSS) \cite{davis2019stochastic} - an online algorithm for weakly convex objective functions, and four Stochastic DCA schemes (SDCA) \cite{an2020stochastic} proposed for nonconvex, nonsmooth DC programs.
\subsubsection{Experiment setup and results}
\label{exp_res}

The numerical experiments comprise of four parts. The first experiment is the comparative experiment between the proposed algorithms with two versions of PSS and four SDCA schemes, the second experiment studies our algorithms' behaviors when the DC decomposition of the problem varies, the third experiment compares between convex solvers for solving subproblems, and the fourth experiment studies the adaptive capacity of osDCA schemes.

In the first experiment, we compare osDCA schemes with two versions of PSS (constant stepsize policy and diminishing stepsize policy) and four SDCA schemes. Firstly, we ran the PSS with many different stepsizes and observed its performance in order to choose a proper range to find a good stepsize. We then ran the PSS with the constant stepsize in $\{0.001,0.005,0.01,0.015,0.02\}$ and found that the stepsize $0.005$ consistently gives good performance on four validation sets. About the diminishing stepsize $\alpha_k=c/k$, we ran PSS with $c$ being chosen in $\{4,5,\ldots,11,12\}$ and found that $c=8$ achieves good performance on all four datasets. For the four SDCA schemes, it should be stressed that SDCA1 and SDCA3 require the first DC component of the objective to be explicitly defined, meanwhile, SDCA2 and SDCA4 can handle the unknown first DC component. Therefore, we apply SDCA1 and SDCA3 to (\ref{firstDC}) with $\lambda = 10^{-6}$ that yields good results; meanwhile, SDCA2 and SDCA4 are applied to (\ref{secondDC}) where $L=1$. We use the sequence of equal weights for all four SDCA schemes. On the other hand, based on the theoretical analysis, the parameters of osDCA schemes are chosen as follows. For the osDCA-1, we choose the sequence of sample sizes as $n_k=k^2$, and $\lambda=1$ which is a neutral number and results in a good performance over four datasets. For the osDCA-2, the sequence of sample sizes is chosen as $n_k=k^3$, the Lipschitz smoothness constant $L=1$ and the tolerance error in solving subproblems $\epsilon = 10^{-5}.$

As a preprocessing step, each training dataset is randomly shuffled before each run. Then, the mentioned algorithms perform one pass through each training dataset and automatically terminate when the training dataset is used up. The starting points are also randomly initialized in $S$. The performance of our algorithms are measured on the validation set to guarantee their generalization capability. To enhance visualization, we first find the ``optimal solution" $w^*$ on the validation set by running deterministic DCA. We then report the suboptimality graph $F(w_n)-F(w^*)$ (under the log-scale) averaging over $20$ runs. Furthermore, we classify osDCA-1, SDCA1, SDCA3 in one group and osDCA-2, SDCA2, SDCA4 in another group (since the former three use the DC decomposition (\ref{firstDC}) and the latter three use (\ref{secondDC})) to plot them in two different figures.

All experiments are performed on a PC Intel(R) Core(TM) i7-8700 CPU @3.20GHz of 16 GB RAM.

Figures \ref{fig1} and \ref{fig1b} illustrate the performance of osDCA schemes compared with SDCA schemes and the PSS with constant stepsize and diminishing stepsize. 

\textit{Comparisons between osDCA schemes and PSS.} Our algorithms take a very short amount of time to pass through the training sets while obtaining really small suboptimality values, say $10^{-4} \sim 10^{-5}.$  In contrast, the PSS with constant stepsize struggles to reach the optimal solution and exhibits the well-known fluctuation behavior with the suboptimality varying around $10^{-3} \sim 10^{-4}.$ On the other hand, PSS with diminishing stepsize performs very well and obtains similar suboptimality as osDCA schemes, where the differences (i.e., $F_{\text{val}}(w_{\text{pss}})-F_{\text{val}}(w_{\text{osdca}})$, where $F_{\text{val}}$ is the objective function measured on the validation set, $w_{\text{pss}}$ and $w_{\text{osdca}}$ are solutions found by PSS and osDCA, respectively) between this PSS and osDCA-1 (resp. osDCA-2) range from $-6.18 \times 10^{-6}$ to $-1.45 \times 10^{-6}$ (resp. from $-3.84 \times 10^{-7}$ to $5.89 \times 10^{-6}$).  To obtain this result, osDCA-1 (resp. osDCA-2) is $2.7 \sim 18.4$ (resp. $1.7 \sim 32.3$) times faster the PSS with diminishing stepsize. 

\textit{Comparisons between osDCA and SDCA.} The differences (i.e., $F_{\text{val}}(w_{\text{sdca}})-F_{\text{val}}(w_{\text{osdca}})$) between SDCA1 (resp. SDCA3) and osDCA-1 vary from $-2.23 \times 10^{-5}$ to $-2.44 \times 10^{-6}$ (resp. $-2.25 \times 10^{-5}$ to $-2.44 \times 10^{-6}$). Wall-clock time for osDCA-1 to perform one pass through the training datasets is $2.7 \sim 18.5$ ($6.1 \sim 14.5$) times shorter than SDCA1 (resp. SDCA3). The differences (i.e., $F_{\text{val}}(w_{\text{sdca}})-F_{\text{val}}(w_{\text{osdca}})$) between SDCA2 (resp. SDCA4) and osDCA-2 are from $-7.81 \times 10^{-7}$ to $5.53 \times 10^{-4}$ (resp. $-1.04 \times 10^{-5}$ to $-8.95 \times 10^{-7}$). Moreover, osDCA-2 makes one pass through the training datasets $5.5 \sim 24.3$ (resp. $4.3 \sim 17.6$) times faster than SDCA2 (resp. SDCA4). We also observe that, at the time osDCA schemes terminate, they usually obtain smaller optimality values than SDCA schemes.

\begin{figure}
     \centering
        \subfigure[\texttt{SensIT Vehicle}]{	
	\includegraphics[width=0.233\textwidth]{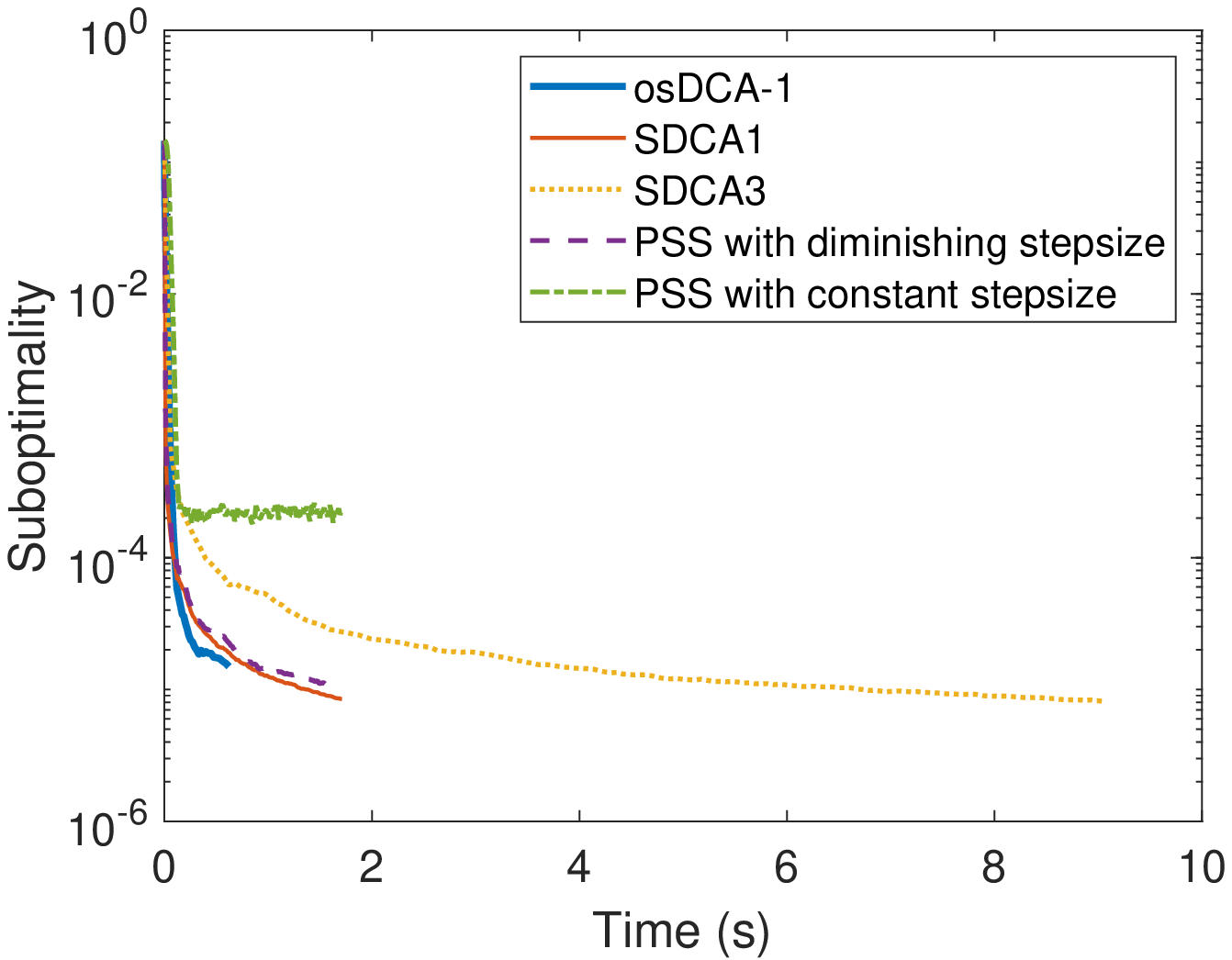} 
	}
	\hspace{-10pt}
	\subfigure[\texttt{shuttle}]{	
	\includegraphics[width=0.233\textwidth]{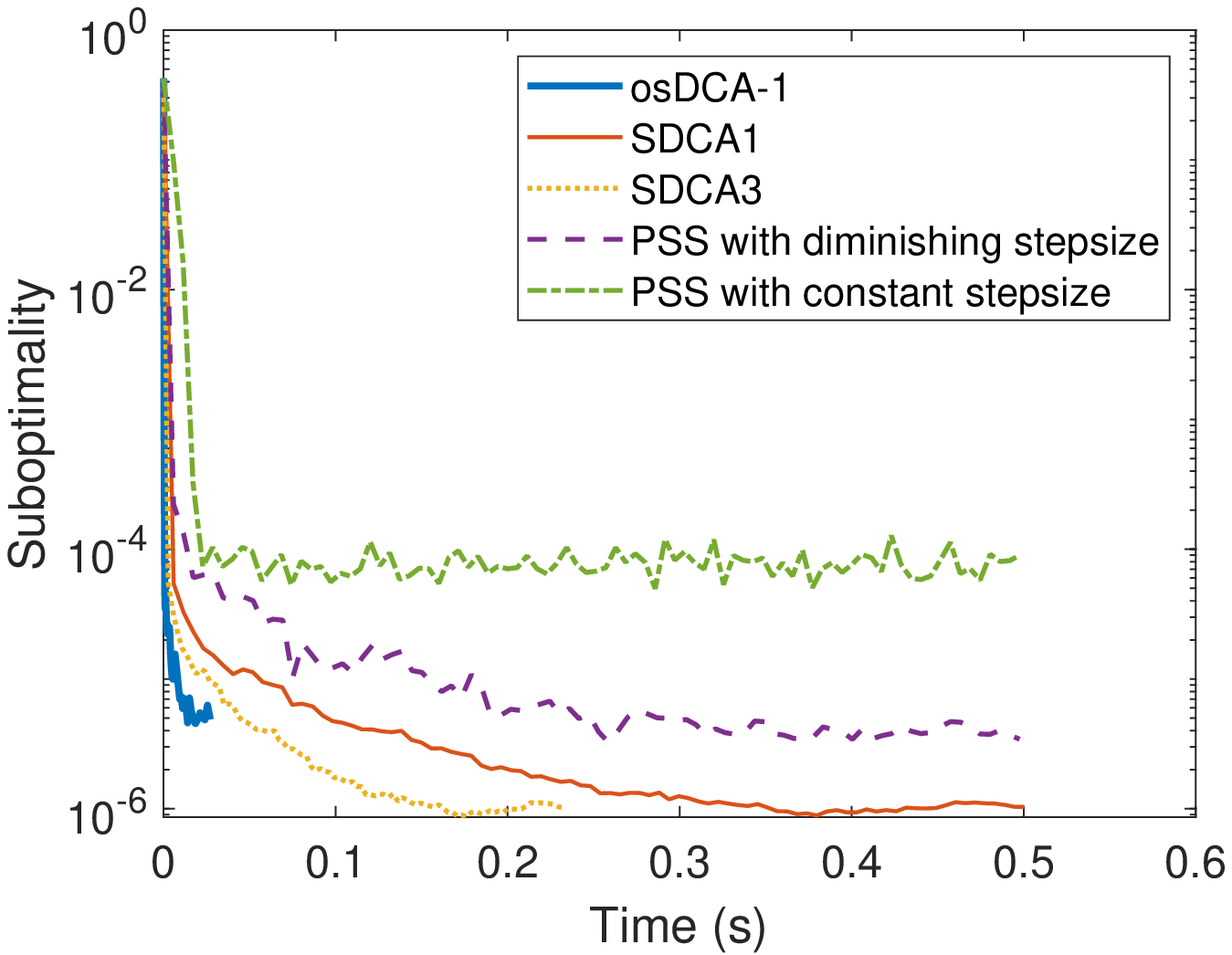} 
	}
	\hspace{-10pt}
	\subfigure[\texttt{letter}]{	
	\includegraphics[width=0.233\textwidth]{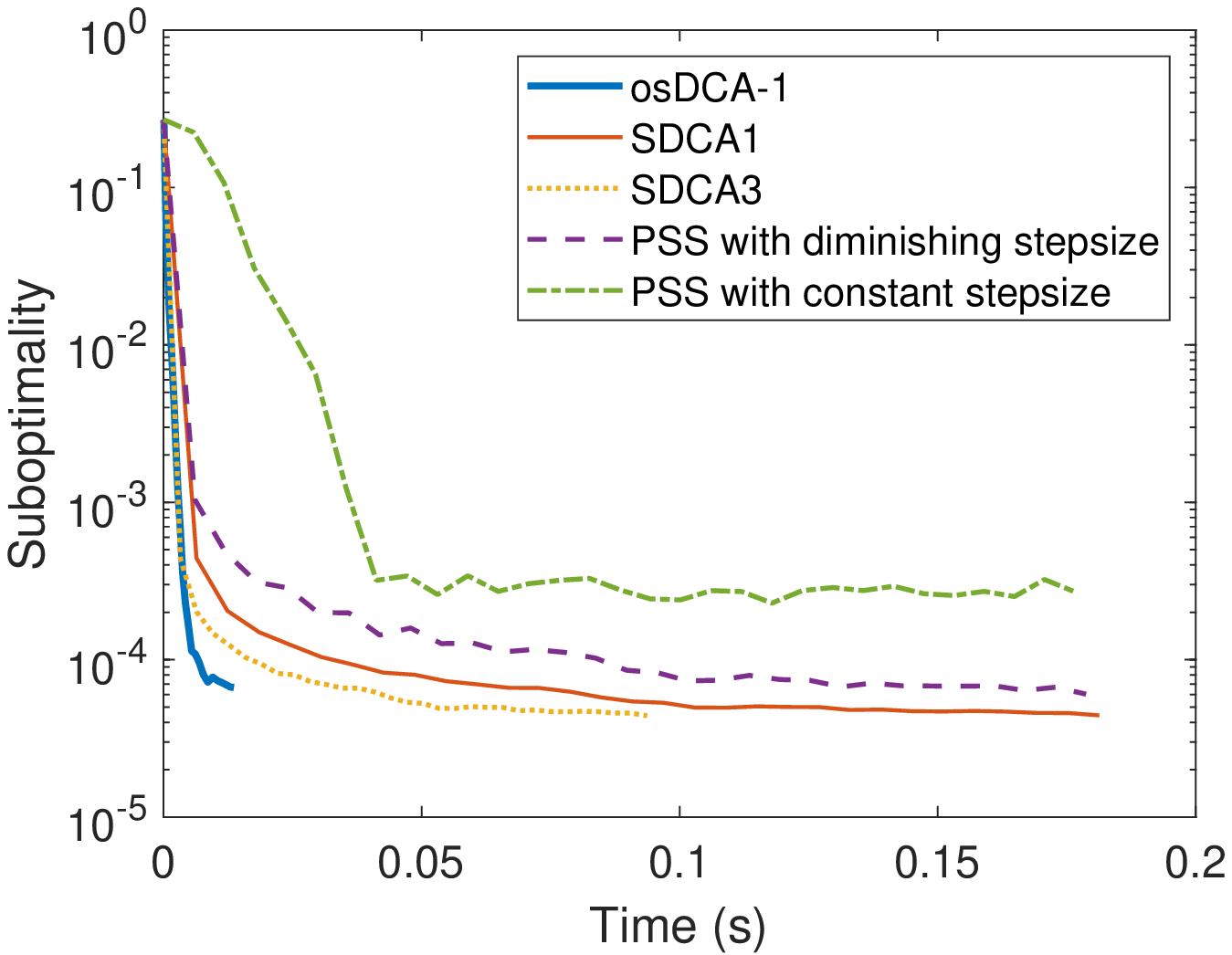} 
	}
	\hspace{-10pt}
	\subfigure[\texttt{YearPredictionMSD}]{	
	\includegraphics[width=0.233\textwidth]{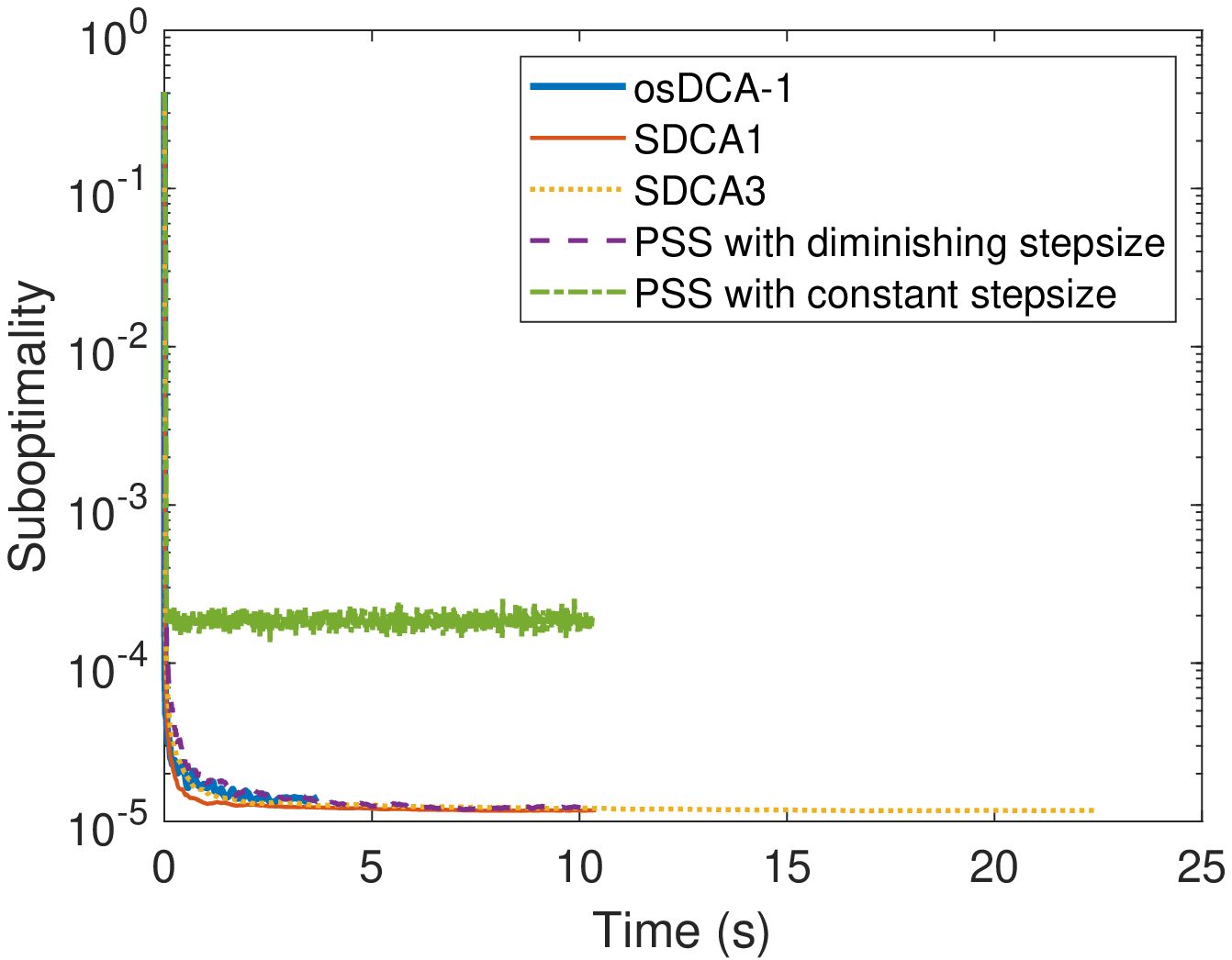} 
	}
     \caption{The performance of osDCA-1 compared with SDCA1, SDCA3 and two versions of PSS.}
     \label{fig1}
\end{figure}

\begin{figure}
     \centering
        \subfigure[\texttt{SensIT Vehicle}]{	
	\includegraphics[width=0.233\textwidth]{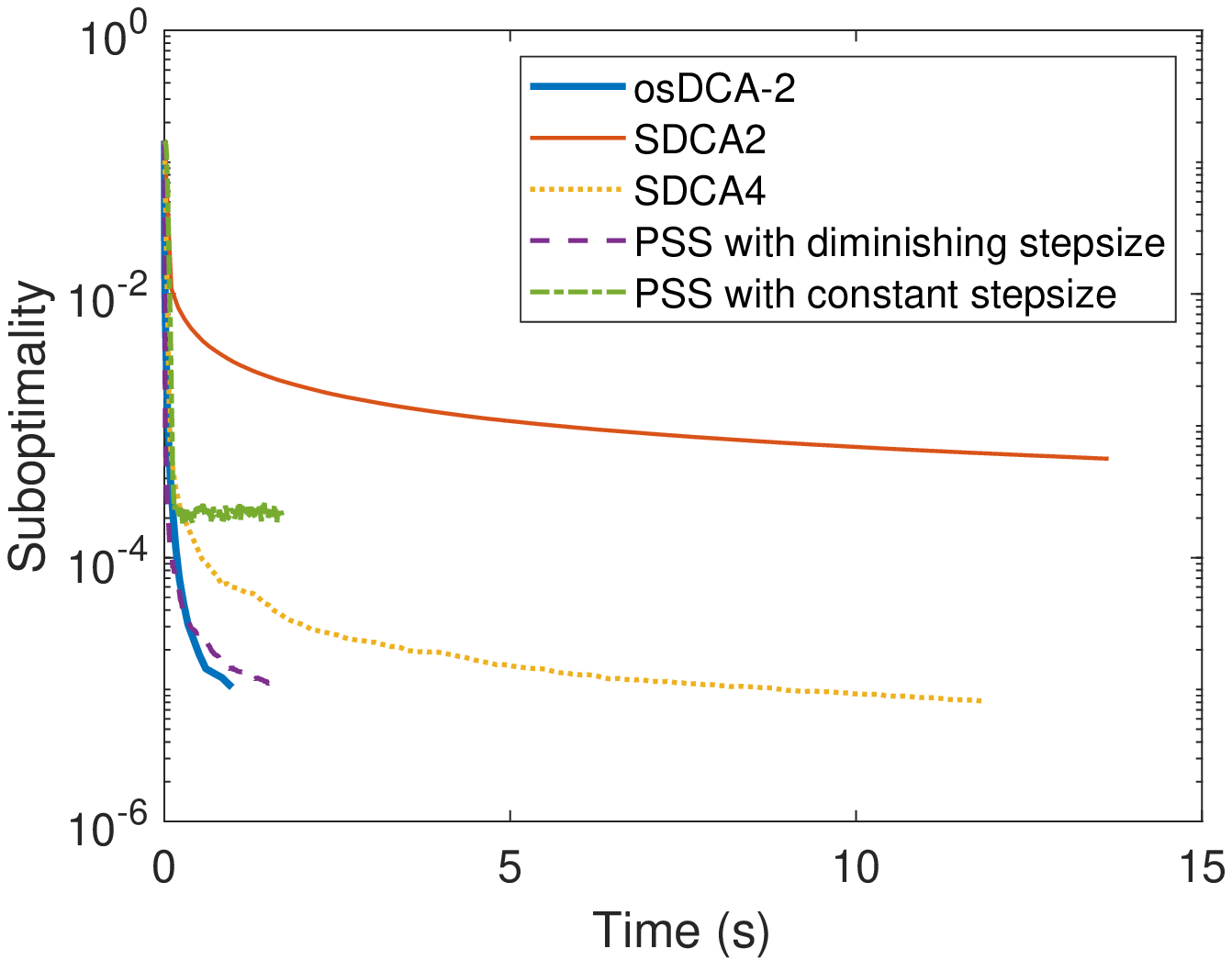} 
	}
	\hspace{-10pt}
	\subfigure[\texttt{shuttle}]{	
	\includegraphics[width=0.233\textwidth]{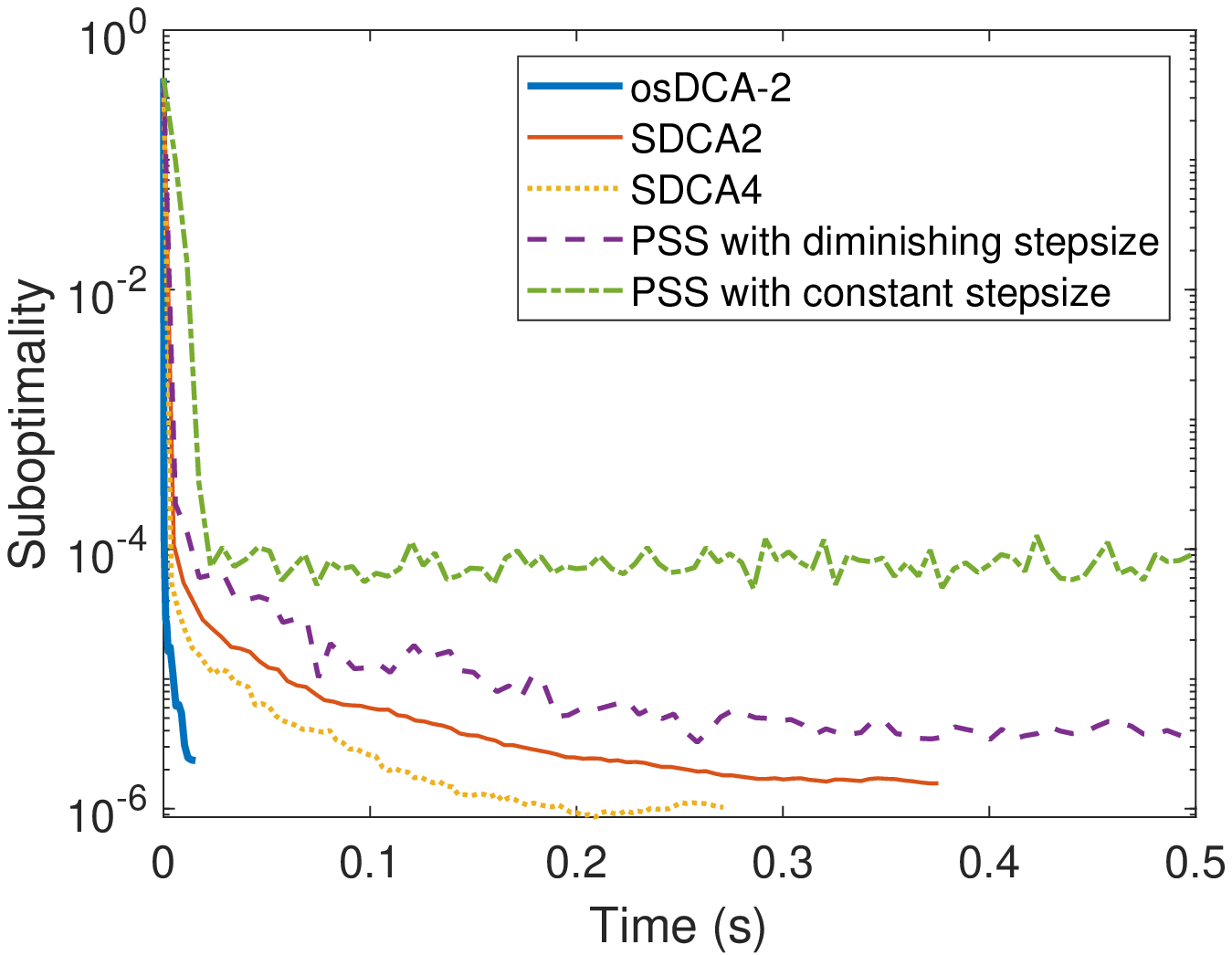} 
	}
	\hspace{-10pt}
	\subfigure[\texttt{letter}]{	
	\includegraphics[width=0.233\textwidth]{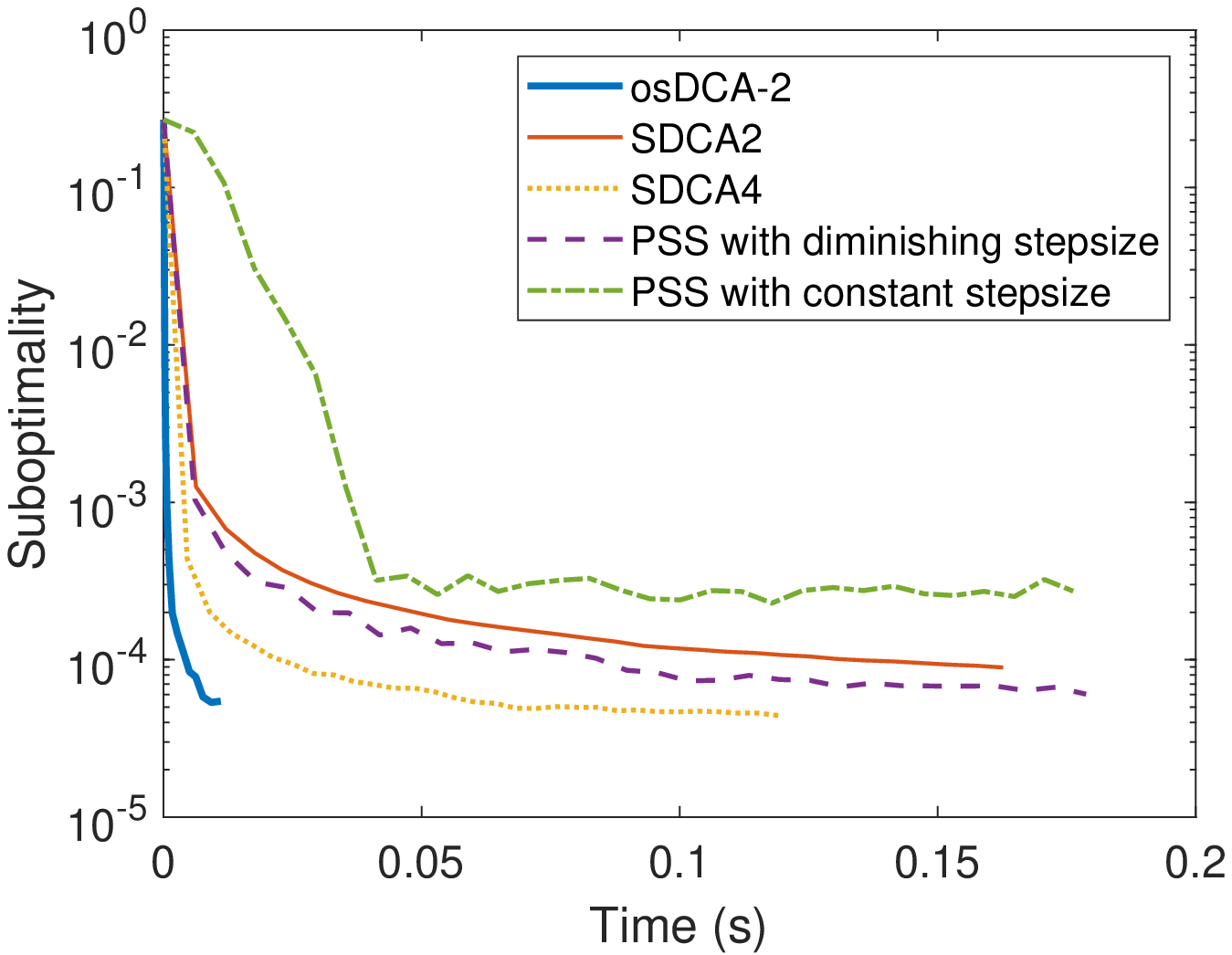} 
	}
	\hspace{-10pt}
	\subfigure[\texttt{YearPredictionMSD}]{	
	\includegraphics[width=0.233\textwidth]{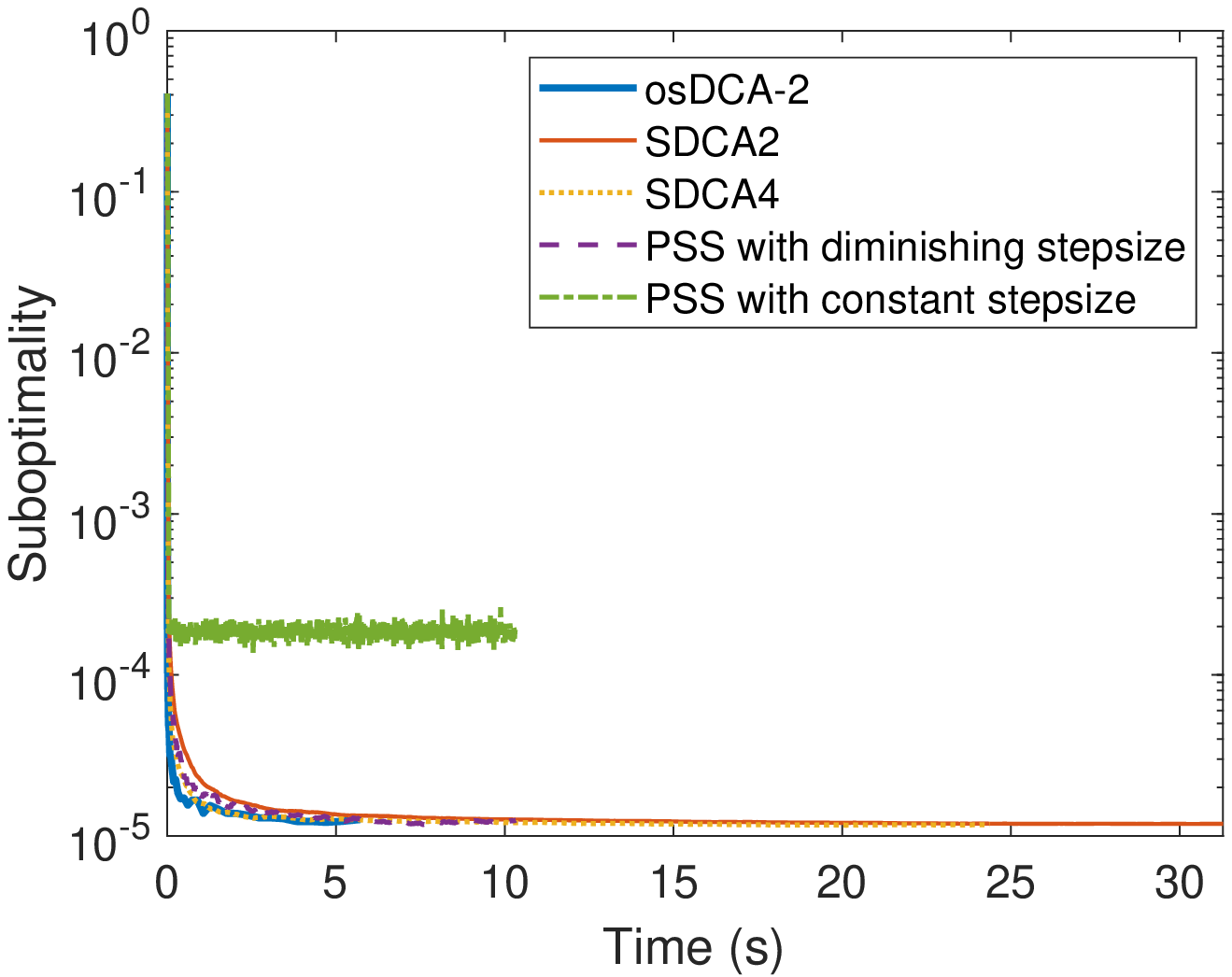} 
	}
     \caption{The performance of osDCA-2 compared with SDCA2, SDCA4 and two versions of PSS.}
     \label{fig1b}
\end{figure}

Furthermore, it is well-known that there are two main factors needed to be carefully considered when designing any DCA (or its variants), namely the DC decomposition of the problem and the convex solver for subproblems. Therefore, we consider the following experiments to study our proposed algorithms' behaviors within these two mentioned perspectives.

In the second experiment, our aim is to study the behavior of osDCA-1 when $\lambda$ varies (change the DC decomposition of the problem). It is observed that, to surely fulfill the strong convexity condition $\rho_G+\rho_H>0$, we add the regularization term $\lambda \Vert \cdot \Vert^2$ to both $G$ and $H$ components. A natural question raised is that: suppose $H$ is already strongly convex, will we obtain some ``optimal" performance if we do not use this regularization term? This curiosity motivates us to perform the osDCA-1 scheme with DC decomposition $g(w,z)=0, h(w,z) = \frac{1}{2} \langle w,z \rangle^2.$ Before presenting the experimental results, let us discuss a little bit about the condition $\rho_H>0$ in this case. We know that this condition does not always hold and it is equivalent to $\mathbb{E}(Z Z^{\top})$ being positive definite. By definition, the positive definiteness of $\mathbb{E}(Z Z^{\top})$ is equivalent to $\mathbb{E}\left((w^{\top} Z)^2 \right) > 0, \forall w \neq 0.$ Therefore, this condition is violated if there exists $w_0 \neq 0$ such that $\mathbb{E}((w_0^{\top}Z)^2)=0$, or equivalently $w_0^T Z = 0$ almost surely. In other words, the condition $\rho_H>0$ does not hold if there is a perfectly linear dependence between features of the random vector $Z$.

 Figure \ref{fig2} shows the behaviors of osDCA-1 with different $\lambda>0$ and an extreme case where $\lambda=0$ on the \texttt{YearPredictionMSD} dataset. We observe that, the optimal performance of osDCA-1 is achieved at some moderate values of $\lambda$, say, from $1$ to $5$. Besides, the quality of the performance is not monotone with respect to $\lambda$. With large value of $\lambda$, osDCA-1 somehow gets stuck at the beginning. The performance of osDCA-1 is gradually improved as $\lambda$ decreases up to a certain value, and then the performance slightly deteriorates as $\lambda$ continues to approach $0$.

\begin{figure}
\centering
\includegraphics[width=0.38\textwidth]{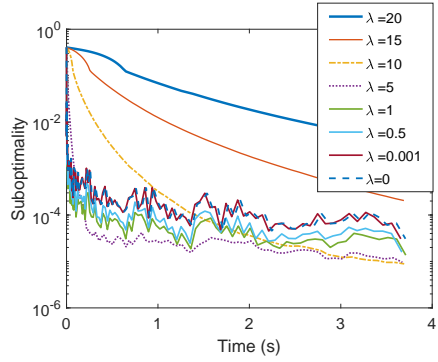}
\caption{Performance (one run) of osDCA-1 when $\lambda>0$ varies and when $\lambda=0$}
\label{fig2}
\end{figure}

In the third experiment, we study the performance of osDCA-2 with different convex solvers for subproblems. To be specific, beside the (deterministic) DCA used in the osDCA-2 scheme, we want to use the industrial CPLEX for solving the convex subproblems. Figure \ref{fig3} shows the difference between osDCA-2 using deterministic DCA and CPLEX for solving convex subproblems. It is observed from the figure that while the suboptimality values of these two algorithms are similar, osDCA-2 using DCA for the convex subproblem is faster than osDCA-2 with CPLEX.

\begin{figure}
     \centering
        \subfigure[\texttt{SensIT Vehicle}]{	
	\includegraphics[width=0.233\textwidth]{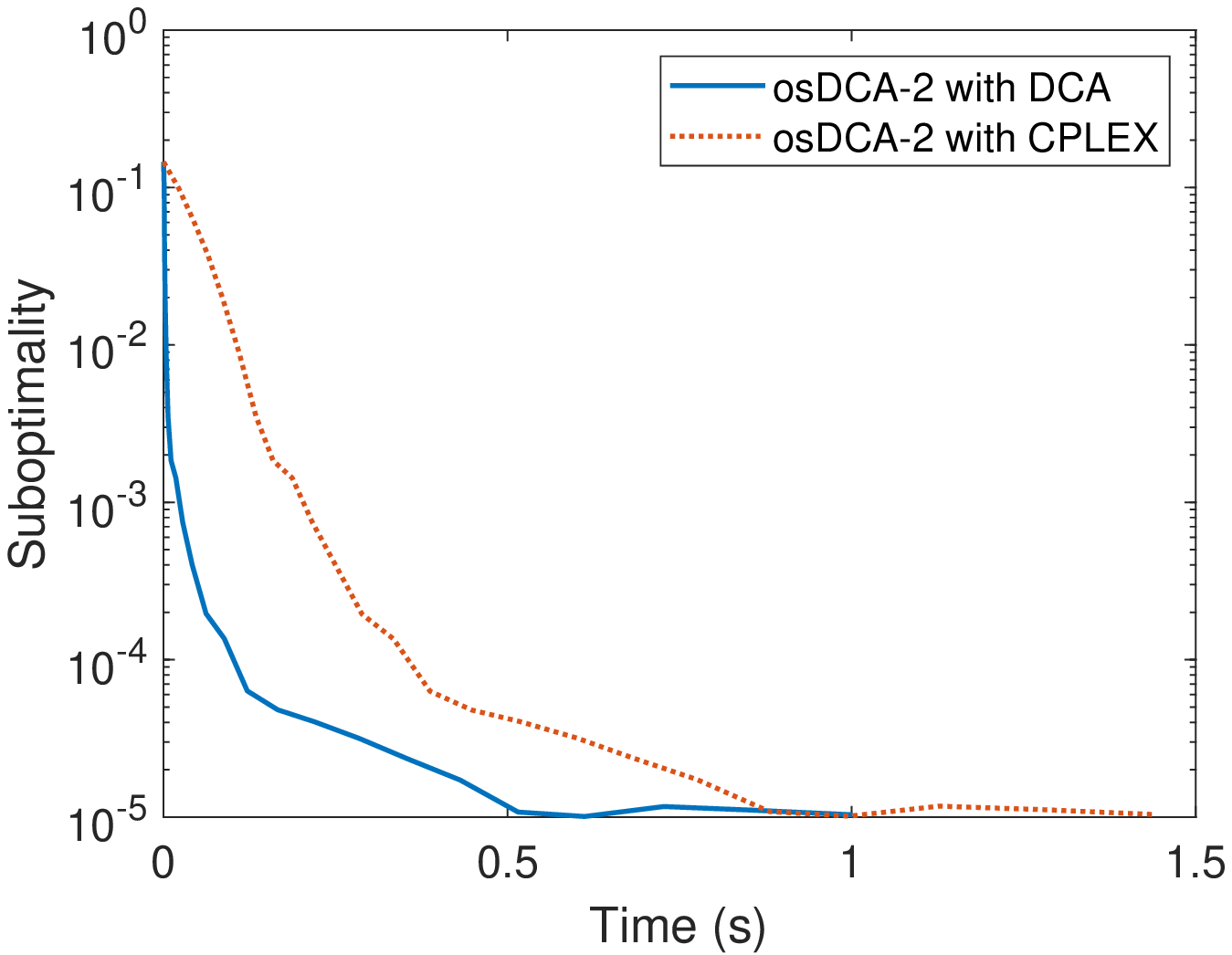} 
	}
	\hspace{-10pt}
	\subfigure[\texttt{shuttle}]{	
	\includegraphics[width=0.233\textwidth]{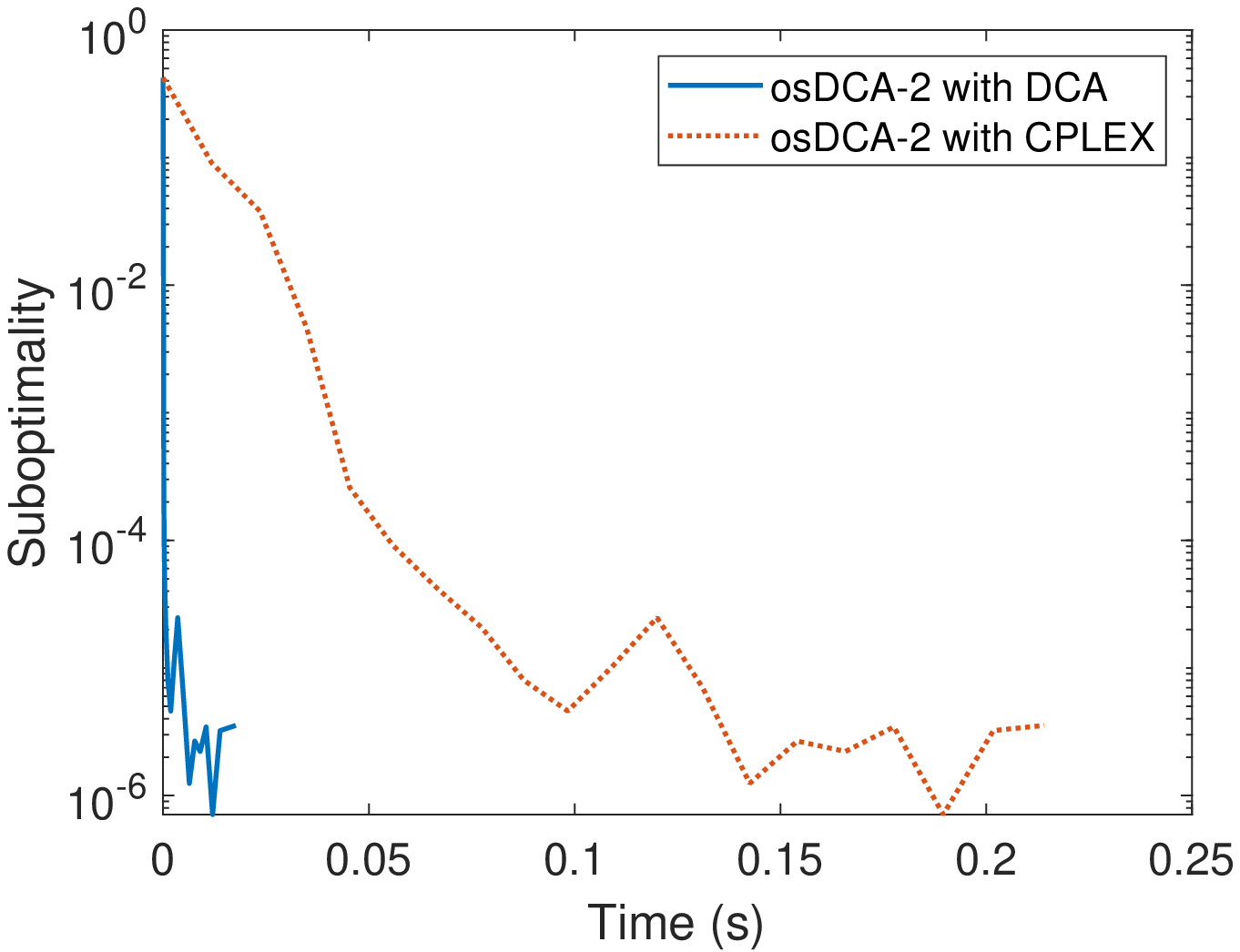} 
	}
	\hspace{-10pt}
	\subfigure[\texttt{letter}]{	
	\includegraphics[width=0.233\textwidth]{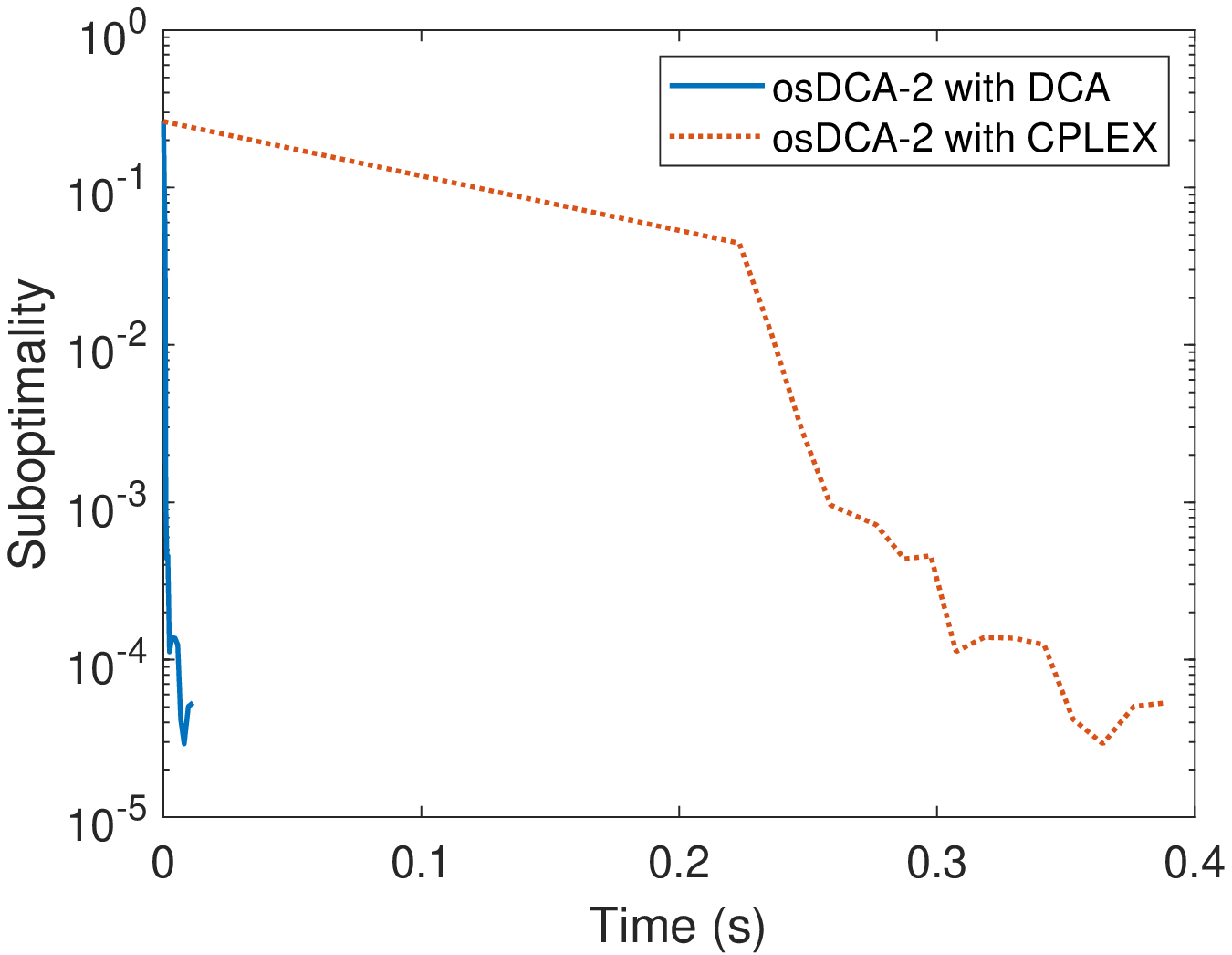} 
	}
	\hspace{-10pt}
	\subfigure[\texttt{YearPredictionMSD}]{	
	\includegraphics[width=0.233\textwidth]{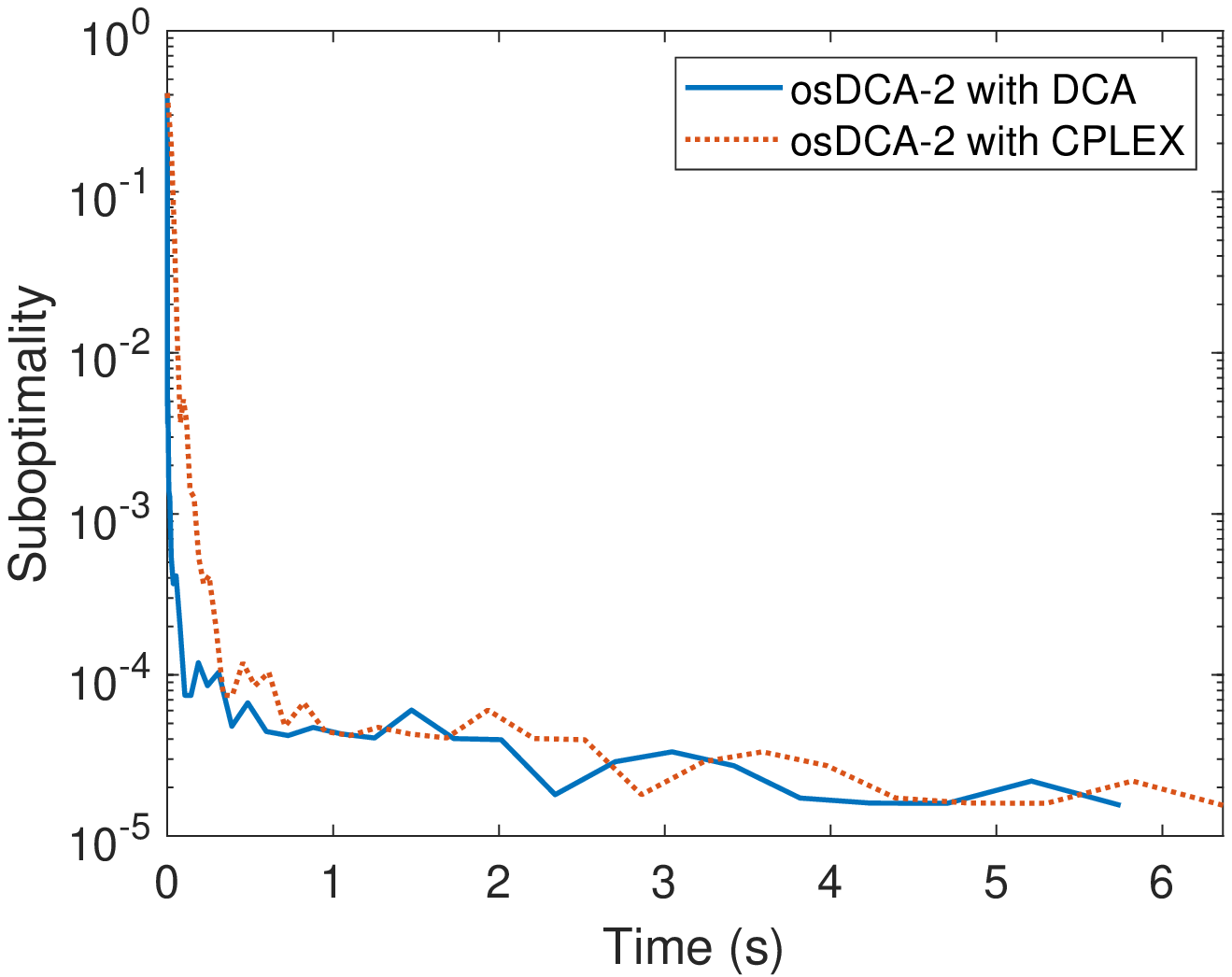} 
	}
     \caption{The performance (one run) of osDCA-2 with two different convex solvers: the DCA and CPLEX}
     \label{fig3}
\end{figure}

In the last experiment, we study the adaptive capacity of osDCA schemes compared with SDCA schemes when there is an abrupt change in the distribution of $Z$. We describe the context of the problem as follows. We are receiving streaming data from an unknown distribution (the data is - in fact - realizations of $Z$).  At a certain time, suppose that there is a real-world event that makes the distribution of $Z$ change ($Z$ becomes some $Z'$). We do not know this event (and hence, the change of $Z$ is also unknown to us) and continue to receive streaming data from the changed distribution. From that time, we want to solve (\ref{eq:ori}) with $Z$ being replaced by $Z'$ since the new random variable $Z'$ is more relevant than $Z$.

To this end, we generate a synthetic dataset as follows. The dataset consists of two subdatasets representing data collected before and after the abrupt change. The first subdataset includes a training set ($200000 \times 500$) and a validation set ($500000 \times 500$) that are generated from multivariate normal distribution with a mean vector $0$ and a positive definite covariance matrix. Then, we change the covariance matrix and generate the second subdataset consisting a training set ($200000 \times 500$) and a validation set ($200000 \times 500$). All data is then normalized as $\Vert z_i \Vert = 1.$ We concatenate two training sets to create one unified training set in order to feed to the algorithms. Before the change, we measure the performance of each algorithm on the first validation set, and after the change, we use the second validation set. Figure \ref{figx} shows the average results of $20$ runs, here we separate the results into two subfigures because the running times of SDCA2, SDCA3, SDCA4 are remarkably longer than osDCA-1, osDCA-2, and SDCA1. The numerical results confirm the adaptive capacity of osDCA schemes over SDCA schemes. Indeed, after the abrupt change, osDCA schemes quickly regain suboptimality values that are as good as the ones obtained before the change. Meanwhile, SDCA schemes barely adapt to the change and decrease the suboptimality slowly.

\begin{figure}
     \centering
        \subfigure{	
	\includegraphics[width=0.35\textwidth]{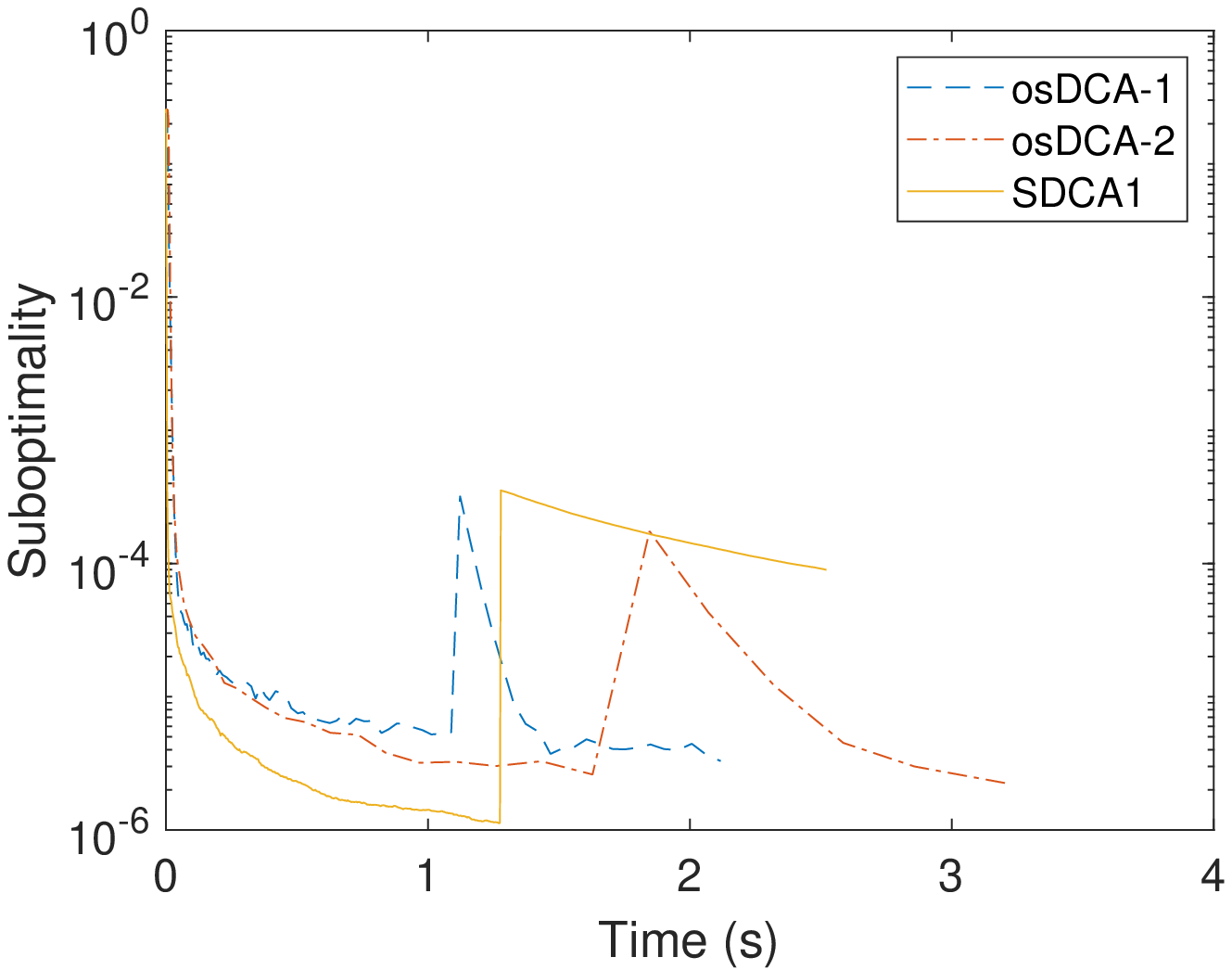} 
	}
	\hspace{-10pt}
	\subfigure{	
	\includegraphics[width=0.35\textwidth]{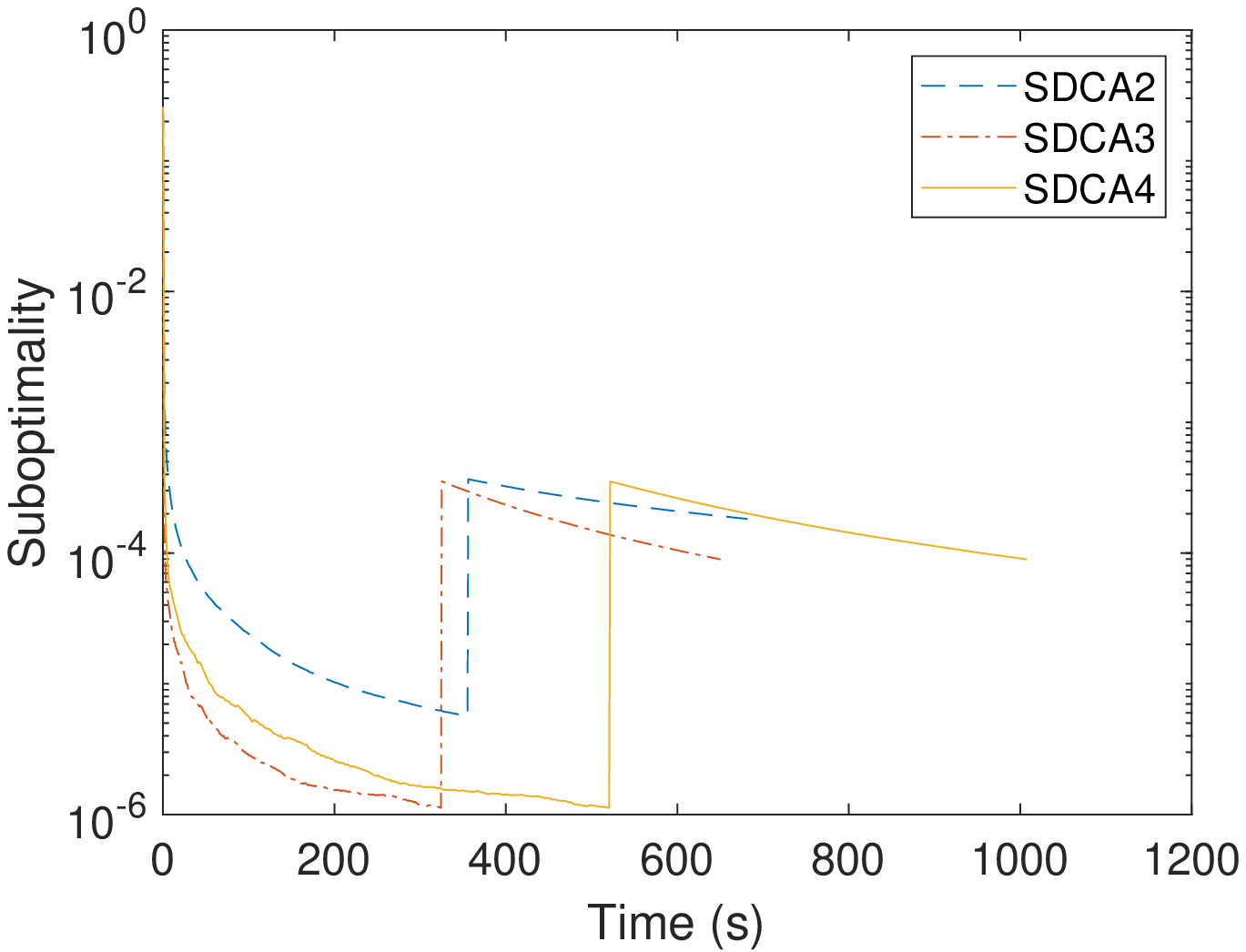} 
	}
     \caption{The adaptive ability of osDCA schemes over SDCA schemes}
     \label{figx}
\end{figure}

\section{Conclusion}
We have designed three online stochastic algorithms based on DCA to handle stochastic nonsmooth, nonconvex DC programs. The first scheme stochastically approximates both DC components; meanwhile, the other two are designed for the context that one of two DC components can be directly computed. The theoretical properties of the proposed algorithms are rigorously studied, and the almost sure convergence to critical points is established. As online stochastic algorithms, the osDCA schemes gain a competitive edge when dealing with streaming data. The benefits of osDCA schemes include remedying storage burden and the ability to adapt to new changes of data distribution. On the other hand, it is well-known that the variance of stochastic estimators of online stochastic algorithms is high, which creates difficulties in the convergence analysis, especially in nonconvex and nonsmooth settings. Our algorithms' convergence results hold thanks to the increase of sample sizes. Moreover, the rate of this increase is determined based on the Rademacher complexity of the family of functions $\{g(\cdot,z): z \in \Xi\}$. Nevertheless, such complexity is not always easy to compute. In future works, we would like to improve this condition and provide a better rate.

On the other hand, to study the practical behaviors of the proposed algorithms, we conduct numerical experiments on the expected problem of PCA. We consider streaming data that comes from an unknown distribution. The numerical experiments justify the proposed algorithms' efficiency. Indeed, the proposed osDCA schemes obtain good solutions within a short time. In addition, the adaptive capacity of osDCA schemes have been confirmed: after a change of the data distribution, our algorithms quickly adapt to the new distribution. As a comparison, SDCA schemes do not have this ability. Further experimental insights confirm the importance of choosing the DC decomposition for the considered problem and the convex solver for subproblems. It has been shown that the (deterministic) DCA is a very efficient and robust convex solver in our experiments.

\bibliographystyle{plain}
\bibliography{osDCA}  






\end{document}